\documentclass[reqno]{amsart}
\usepackage{amsmath,amsthm,amssymb,amsfonts,xspace,tensor}
\usepackage[english]{babel}
\usepackage[usenames,dvipsnames]{xcolor}
\usepackage{xcolor}
\usepackage[utf8]{inputenc}
\usepackage{amssymb,xspace,tensor}
\usepackage[all]{xy}
\usepackage{dsfont}
\usepackage[bookmarksnumbered,colorlinks]{hyperref}
\usepackage{graphicx}
\usepackage{enumerate}
\usepackage{mathrsfs}
\usepackage[normalem]{ulem}
\usepackage[colorinlistoftodos,italian]{todonotes}

\newtheorem{thm}{Theorem}[section]
\newtheorem{prop}[thm]{Proposition}
\newtheorem{lem}[thm]{Lemma}
\newtheorem{cor}[thm]{Corollary}
\theoremstyle{definition}
\newtheorem{dfn}[thm]{Definition}
\theoremstyle{remark}
\newtheorem{rmk}[thm]{Remark}

\def\bal#1\eal{\begin{align}#1\end{align}}              
\def\baln#1\ealn{\begin{align*}#1\end{align*}}          
\def\bml#1\eml{\begin{multline}#1\end{multline}}        
\def\bmln#1\emln{\begin{multline*}#1\end{multline*}}  
\def\bga#1\ega{\begin{gather}#1\end{gather}}
\def\bgan#1\egan{\begin{gather*}#1\end{gather*}}

\newcommand{\R}{\ensuremath{\mathbb R}\xspace}

\newcommand{\eps}{\ensuremath{\epsilon}\xspace}

\newcommand{\beq}{\begin{equation}}
\newcommand{\eeq}{\end{equation}}
\newcommand{\bere}{\begin{rmk}}
\newcommand{\ere}{\end{rmk}}

\title[On Finsler spacetimes with a timelike Killing vector field]{On Finsler spacetimes with a timelike \\ Killing vector field} 
\author[E. Caponio]{Erasmo Caponio}
\address{Dipartimento di Meccanica, Matematica e Management, \hfill\break\indent
Politecnico di Bari, Via Orabona 4, 70125, Bari, Italy}
\email{caponio@poliba.it}
\author[G. Stancarone]{Giuseppe Stancarone}
\address{Dipartimento di Matematica\hfill\break\indent Universit\`a degli Studi di Bari, 
Via  Orabona 4, 70125 Bari, Italy}
\email{giuseppe.stancarone@uniba.it}

\keywords{Finsler spacetime, stationary spacetime, causality}

\begin{document}
\begin{abstract}
We study Finsler spacetimes and Killing vector fields taking  care of the fact that the generalised metric tensor associated to the Lorentz-Finsler function $L$ is in general well defined only on a subset of  the slit tangent bundle. We then  introduce a new class of  Finsler spacetimes endowed with a timelike Killing vector field that we call {\em stationary splitting Finsler spacetimes}. We  characterize when a Finsler spacetime with a timelike Killing vector field is locally a stationary splitting. Finally, we show that  the  causal structure of a stationary splitting is  the same of one of two Finslerian static spacetimes naturally associated to the stationary splitting.   
\end{abstract}
\maketitle
\section{Introduction}
The main feature of Finsler geometry is that the associated {\em generalised} metric tensor $\tilde g$ (also called {\em fundamental tensor})  has a dependence on  the directions. More precisely,  at any point $p$ of a spacetime $\tilde M$ there are infinitely many scalar products $\tilde g_v$, one for each direction $v$ where $\tilde g$ is defined.  
In this respect, a gravitation theory based on Finsler geometry is a metric one and it can  include, as its  isotropic case,   general relativity.  There have indeed been several works in relativistic physics where   Finsler geometry is used.\footnote{For applications of Finsler geometry  to non-relativistic physics and  biology we recommend  \cite{AnInMa93}.} We recall here the pioneering work  of  G. Randers \cite{Rander41} about asymmetry of time intervals, where it is introduced a class of Finsler metrics,   nowadays called Randers metrics, together with  its connection  with 5-dimensional Kaluza-Klein theory; the work of G. Y. Bogoslovsky on  Lorentz symmetry violation  (see e.g. \cite{Bogosl77a, Bogosl77b, Bogosl94}) and  the finding by G.W. Gibbons et al. \cite{GiGoPo07} that general very special relativity is the group of transformations that leave invariant  a Finsler metric introduced by Bogoslovsky; the work of H.E. Brandt about maximal proper acceleration (see e.g. \cite{Brandt1999}) where generalised metric on the tangent bundle of the spacetime are used; the extension of the Fermat's principle to Finsler spacetime by V. Perlick \cite{perlick06}, motivated by optics in  anisotropic non-dispersive media; the applications to quantum gravity by F. Girelli et al. \cite{GiLiSi2007}.  More recently, mathematical models where Finsler geometry replaces Lorentzian one have been considered in  gravitation, see e.g. \cite{amir, Bar12, FusPab16, KoStSt12, LiChang, Minguz15b, PfeWol11},  cosmology \cite{HohPfe17, papa, stavac,vacaru12}  and in the so-called Standard Model Extension (see e.g. \cite{Colladay, Koste, Russell, Russel15, shreck}).

In this work we study  Finsler spacetimes endowed with a timelike Killing vector  field  and we introduce  a new class of Finsler spacetimes that  can be viewed  as a Finslerian extension  of the class of the   standard stationary Lorentzian manifolds (see, e.g., \cite{mas, FoGiMa95, JavS08}).

A Finsler spacetime is here defined as a smooth, connected, paracompact manifold $\tilde M$ of dimension $n+1$, $n\geq 1$, endowed with a generalised  metric tensor $\tilde g$,  defined on an open subset $A\subset T\tilde M$, having  index $1$ for each $v\in A$    and which is the Hessian w.r.t. to the velocities of a Lorentz-Finsler function $L$ (see Definition~\ref{fst} below). This is a function 
$L:T\tilde M\rightarrow\mathbb{R}$ which is positively homogeneous of degree $2$ in the velocities, i.e. $L(p,\lambda v)=\lambda^{2}L(p,v)$, $\forall\lambda>0$.
The domain $A$ where $\tilde g$ is well defined and has index $1$  is in general, a smooth,  cone subset of $T\tilde M\setminus 0$, where $0$ denotes the zero section in $T\tilde M$.  By ``smooth cone subset'' we mean that  $\tilde\pi(A)=\tilde M$, where $\tilde\pi:T\tilde M\rightarrow \tilde M$ is the natural projection from the tangent bundle $T\tilde M$ to $\tilde M$, and, for every $p\in M$, $A_p:=A\cap T_p\tilde M$ is an open linear cone (without the vertex $\{0\}$) of the tangent space $T_{p}\tilde M$,  i.e. if $v\in A_p$ then $\lambda v\in A_p$ for each $\Lambda>0$. Moreover $A_p$ varies smoothly with $p\in M$ in the sense that $A_p$ is defined by the union of the solutions of a finite number of of systems of inequalities  
\[
\begin{cases}E_{1,k}(p,v)>0\\
\ldots\\
E_{m_k,k}(p,v)>0\end{cases}\]
where, for each $k\in\{1,\ldots,l\}$, $E_{1,k}, \ldots, E_{m_k,k}\colon T\tilde M\to \R$ are $m_k$ smooth functions on $T\tilde M$, positively homogeneous of degree $1$ in $v$.  

In our paper, below Section~\ref{kvf},    $A_p$ will be equal to $T^+_p\tilde M\setminus \mathcal T_p$ or, in some cases, to  $T_p\tilde M\setminus \mathcal T_p$, where $T^+_p\tilde M$ is a half-space in  $T_p\tilde M$ whose boundary is an hyperplane passing through  $\{0\}$ and $\mathcal T_p$ is a one-dimensional  subspace intersecting $T^+_p\tilde M$. We will denote the set $A=\cup_{p\in\tilde M}A_p$ by $T^+\tilde M\setminus \mathcal T$ in the former case and with $T\tilde M\setminus \mathcal T$ in the latter. Indeed, the cone subsets $T^+_p\tilde M\setminus \mathcal T_p$ and $T\tilde M\setminus \mathcal T$ are the  natural candidate domains for a generalised metric tensor $\tilde g$ if one asks for  a Lorentz-Finsler function $L$ on a product manifold $\tilde M=\R\times M$ such that, on $TM$,  $L$ reduces to the square of a classical Finsler metric. In fact, $L$ cannot be smoothly extended 
to  vectors which project on $0$ in  $TM$, because  the square of a classical Finsler metric is not twice differentiable  on zero vectors. 

More generally, as suggested in \cite{LPH},  $L$ could be  smooth only on $T\tilde M\setminus \mathcal Z$ where $\mathcal Z$ is  a zero measure subset in $T\tilde M$.  It is worth to recall that it would be possible  to define $L$ on a cone subset $A$ where $L$  is  negative and, at each point $p\in \tilde M$, $A_p$ is a  convex salient cone  and $L$ is  extendible and smooth on a  cone subset around the set of lightlike vectors $\{v\in T\tilde M\setminus 0:L(v)=0\}$ which defines the boundary of $A$ in $T\tilde M\setminus 0$,  \cite{amir} ($A$ is then the cone subset of the Finslerian future pointing timelike vectors).

Let  us now give some further details about the generalised metric tensor $\tilde g$. Let $A\subset T\tilde M$ be a cone subset as above and
let $\pi\colon A\to \tilde M$ be the restriction of the canonical projection, $\tilde \pi:T\tilde M\to \tilde M$, to  $A$. Moreover, let $\pi^*(T^*\tilde M)$ the pull-back cotangent bundle over $A$. We consider the tensor product bundle $\pi^*(T^*\tilde M)\otimes\pi^*(T^*\tilde M)$  over $A$ and a section $\tilde g:v\in A\mapsto \tilde g_v\in T^*_{\pi(v)}\tilde M\otimes T^*_{\pi(v)}\tilde M$. We say that $\tilde  g$ is {\em symmetric} if $\tilde g_v$ is symmetric for all $v\in A$. Analogously, $\tilde g$ is said {\em non-degenerate} if $\tilde g_v$ is non-degenerate  for each $v\in A$ and its index will be the common index of the symmetric bilinear forms $\tilde g_v$; moreover,  $\tilde g$ will be said homogeneous if, for all $\lambda>0$ and $v\in A$,  $\tilde g_{\lambda v}=\tilde g_{v}$. A a smooth, symmetric, homogeneous, non-degenerate   section $\tilde g$ of the tensor bundle $\pi^*(T^*\tilde M)\otimes \pi^*(T^*\tilde M)$ over $A$ will be said a {\em generalised metric tensor}.
\begin{dfn}\label{fst}
A {\em Finsler spacetime} is a smooth $(n+1)$-dimensional manifold $\tilde M$, $n\geq 1$, endowed with  a generalised metric tensor $\tilde g$, defined on a (maximal) cone subset $A\subset T\tilde M\setminus 0$,  such that  $\tilde g_{(p,v)}$ has index $1$, for each $(p,v)\in A$, and it is the fiberwise Hessian  of a {\em Lorentz-Finsler function} $L$:

(i) $L:T\tilde M \rightarrow\mathbb{R}$,\quad  $L\in C^0(T\tilde M)\cap C^3(A)$, 

(ii) $L(p,\lambda v)=\lambda L(p, v)$, for all $(p,v)\in T\tilde M$, 

(iii)
\begin{equation}\label{tildeg}
\tilde g_{(p,v)}(u_1,u_2):=\frac 1 2 \frac{ \partial^{2}L}{\partial s_1\partial s_2}(p,v+s_1u_1+s_2u_2)|_{(s_1,s_2)=(0,0)},
\end{equation}
for all $(p,v)\in A$. Moreover, there exists a smooth vector field $Y$ such that  $Y_p\in  \bar{A_p}$ and $L(p,Y_p)<0$, for all $p\in \tilde M$, where $\bar{A_p}$ is the closure of $A_p$ in $T_p\tilde M\setminus \{0\}$. 

We denote a Finsler spacetime by  $(\tilde M, L)$; in some circumstances,  to emphasize that  $\tilde g$ is defined and has index $1$ only on $A\subset T\tilde M\setminus 0$, we denote it by $(\tilde M,L,A)$.      
\end{dfn}
\bere\label{good}
Observe  that, 
whenever $\tilde g$ is defined on  $T\tilde M\setminus 0$, Definition~\ref{fst} coincides with that in \cite[Def. 3]{Minguzzi}.
We emphasize that $A$ has to be intended as the maximal open domain in $T\tilde M\setminus 0$ where $\tilde g$ is well defined and has index $1$. We do not assume a priori that the connected component of $\bar A_p$ that contains $Y_p$  is convex and that  all the lightlike vectors ($v\in T\tilde M\setminus 0$, such that $L(v)=0$) in such a component belong also to $A_p$; anyway both these properties should hold for obtaining reasonable  local and global causality properties  (see \cite{Minguzzi, amir, Minguz15b}) and indeed they are satisfied by the class of stationary  Finsler spacetimes that we introduce  below. On the other hand,
some Finslerian models do not satisfy the above second  requirement: for example, in (deformed) Very Special Relativity minus the square of the line element  (see \cite{GiGoPo07, FusPab16}) is given by
\[L(v):=-(-g(v,v))^{1-b}(-\omega(v))^{2b},\]
 where $g$ is a Lorentzian metric on $\tilde M$ admitting a global smooth timelike vector field $Y$, which gives to $(\tilde M,g)$ a time orientation,  and $\omega$ is a one-form on $\tilde M$ which is  equivalent, w.r.t the metric $g$, to a future-pointing lightlike vector $u$. According to the value of the parameter $b\neq 1$, the  fundamental tensor $\tilde g$ of $L$ is   not defined or vanishing at  $u$, which is also lightlike for $L$, while for all  the timelike future-pointing vectors $v$ of $g$, we have that $L(v)<0$. 
\ere
\bere
We could allow more generality by not prescribing the existence of  a Lorentz-Finsler function.
This is a quite popular generalisation of classical Finsler geometry, see, e.g.,   \cite[\S 3.4.2]{AnInMa93} or \cite{Lovas04,MeSzTo03}, and the references therein, where   such  structures  are indeed called  generalised metrics. 
Anyway as observed above and in \cite[Remark 2.11]{CapSta16} the existence of  a {\em good} (in the sense explained in Remark~\ref{good}) Lorentz-Finsler function avoids the occurrence of some causality issues.
\ere  
Henceforth we will omit the dependence from the point on the manifold $\tilde M$, writing simply  $L(v)$, $\tilde g_{v}$, etc., unless to reintroduce it in necessary cases (as in the statement of Theorem~\ref{charstat} where we use the notation $L(z,\cdot)$ to denote the map from $T_z\tilde M\to \R$ obtained from $L$ by fixing $z\in \tilde M$).


\section{Killing vector fields}\label{kvf}
In this section we extend the notion of  Killing vector field to Finsler spacetimes following the approach of \cite{Lovas04}, with the difference  that the base space in our setting is the open subset  $A\subset T\tilde M$, while in \cite{Lovas04} it is the standard one for Finsler geometry, i.e. the slit tangent bundle. We will only consider  Killing vector fields that are vector fields on $\tilde M$  by passing to their complete lifts on $T\tilde M$ and then restricting them to the open base space $A$. Clearly a more general approach is possible by considering {\em generalised vector field}, i.e. sections of $\pi^*(T\tilde M)$, as in \cite{OoYaIs17}. Another interesting approach is given in \cite[\S 2.9]{java}.  

Let us give some preliminary notions.
Let $f\colon \tilde M\to \R$ be a smooth function. The {\em complete lift of $f$ on $T\tilde M$} is the function $f^c$ defined as $f^{c}(v):=v(f)$ for any $v\in T\tilde M$. 
Let now $X$ be a  vector field on $\tilde M$ and set $X^c$ the \emph{complete lift of $X$ to $T\tilde M$}, defined by   
\baln 
&X^c(f \circ \pi):=X(f),\text{ for all smooth functions $f$ on $\tilde M$,}\\
&X^c(f^c):=(X(f))^c.
\ealn
Observe that  if $(x^0,\ldots,x^n)$ are local coordinates on $\tilde M$ and $(x^0,\ldots, x^n,y^0,\ldots, y^n)$  are the induced ones on  $T\tilde M$ (by an abuse of notation we denote the induced coordinate $x^i\circ\tilde \pi$ again by  $x^i$), then $(x^i)^c=y^i$, for all $i=0,\ldots, n$; so it is easy to check that  in local coordinates $(X^c)_{(x,y)}$ is given by
\beq\label{completelift}X^{h}(x)\frac{\partial }{\partial x^{h}}+\frac{\partial X^{h}}{\partial x^{i}}(x)y^{i}\frac{\partial}{\partial y^{h}},\eeq
where we have used the Einstein summation convention; here   $(x,y)\in T\tilde M$ has coordinates $(x^0,\ldots,x^{n},y^0,\ldots, y^{n})$,  and $X^h(x)$, $h=0,\ldots, n$,  are the components of $X$ w.r.t.  $\left(\frac{\partial}{\partial x^h}\right)_{h\in\{0,\ldots,n\}}$.

\begin{rmk}\label{restrictions}
	It is worth to observe that complete lifts on $A$ are well defined by   restricting  functions and fields to the open subset $A$ and, in the following, we will  consider such restrictions and we will denote them, with an abuse of notation, always by $f^c$ and $X^c$.
\end{rmk}
The canonical vertical bundle map between $\tilde \pi^*(T\tilde M)$ and $T(T\tilde M)$ induces  an injective  bundle map $i:\pi^*(T\tilde M)\rightarrow T(A)$;  in local coordinate $(x^i,y^i)$ of $T\tilde M$, if $z_y= z^{i}(x,y)\frac{\partial}{\partial x^i}|_{(x,y)}$, $(x,y)\in A$, it holds
\begin{equation*}
i(z_y)=z^i(x,y)\frac{\partial}{\partial y^i}|_{(x,y)}.
\end{equation*}
Observe that the map $i$ induces also a injective homomorphism between $\mathfrak{X}(\pi)$ and $\mathfrak{X}(A)$, denoted always by $i$, where $\mathfrak{X}(\pi)$ and $\mathfrak{X}(A)$ are the sets of smooth sections of $\pi^*(T\tilde M)$ (over $A$)  and, respectively, of $T(A)$. In analogous way,  a map $j:T(A)\rightarrow\pi^{*}(T\tilde M)$ can be defined as $j(w):=d\pi_y(w)$, for every $w\in T_{y}A$. Observe that $ i(\pi^*(T\tilde M))=\ker j$ and we have the following exact sequence 
\begin{equation*}
0\rightarrow\pi^*(T\tilde M)\xrightarrow{i}T(A)\xrightarrow{j}\pi^{*}(T\tilde M)\rightarrow 0.
\end{equation*}
Thus another homomorphism between $\mathfrak{X}(A)$ and $\mathfrak{X}(\pi)$, denoted always by $j$, is defined and it holds the following exact sequence 
\begin{equation*}
0\rightarrow\mathfrak{X}(\pi)\xrightarrow{i}\mathfrak{X}(A)\xrightarrow{j}\mathfrak{X}(\pi)\rightarrow 0.
\end{equation*}
The \emph{vertical vectors fields } are the elements of $ i(\mathfrak X(\pi))$.
We define the \emph{Lie derivative} relative to any smooth vector field $X$ on $\tilde M$ on the tensor product bundles  of the pull-back bundles $\pi^*(T\tilde M)$ and $\pi^*(T^*\tilde M)$ over $A$ such that:
\begin{equation}\label{LD}
\mathcal{L}_{X}f:=X^c(f),\quad\mathcal{L}_{X} Y:=i^{-1}([X^c,i(Y)]),
\end{equation}
for any smooth function $f$ on  $A$ and any $Y\in \mathfrak X(\pi)$, where  $[\cdot, \cdot]$ is  the Lie bracket on $A$ (recall Remark~\ref{restrictions}).
Then $\mathcal L_X$ is extended to any section of  the tensor product bundles  of the pull-back bundles $\pi^*(T\tilde M)$ and $\pi^*(T^*\tilde M)$ over $A$ by the 
 generalised Willmore's theorem for tensor derivations (see, e.g. 
\cite[\S 1.32]{Szilas03}). Observe that the second equation in (\ref{LD}) is well posed, namely $[X^c,i( Y)]$ is vertical. In fact, it is almost immediate to see that the Lie bracket of any vector field $X^c$ and any vertical vector field is vertical; in local coordinates $(x^i,y^i)$ of $T\tilde M$, if $Y=Y^k(x,y)\frac{\partial}{\partial x^k}$ and $X=X^h(x)\frac{\partial}{\partial x^h}$, we have indeed 
\begin{equation}\label{vLie}
[X^c,i(Y)]=\left(X^h\frac{\partial Y^k}{\partial x^h}+\frac{\partial X^h}{\partial x^i}y^i\frac{\partial Y^k}{\partial y^h}-Y^{h}\frac{\partial X^k}{\partial x^h}\right)\frac{\partial}{\partial y^k}.
\end{equation}
The Lie derivative $\mathcal L_X$ on $\pi^*(T^*\tilde M)\otimes\pi^*(T^*\tilde M)$ is then, 
\begin{equation}\label{Lie}
\mathcal L_{X}\tilde g(Y,Z):=X^c(\tilde g(Y,Z))-\tilde g(\mathcal L_X Y,Z)-\tilde g(Y,\mathcal L_X Z),
\end{equation}
for any $\tilde g\in \pi^*(T^*\tilde M)\otimes\pi^*(T^*\tilde M)$ and for every $Y, Z\in \mathfrak{X}(\pi)$.
Observe that in a local base $\left(\widehat{\frac{\partial }{\partial x^0}},\ldots,\widehat{\frac{\partial }{\partial x^n}}\right)$ of $\mathfrak X(\pi)$,
$\widehat{\frac{\partial }{\partial x^i}}:=\frac{\partial }{\partial x^i}\circ\pi$, for each $i\in\{0,\ldots,n\}$,  we have: 
\begin{eqnarray}\label{jjj}
\nonumber\mathcal L_{X}\tilde g\left(\widehat{\frac{\partial}{\partial x^l}},\widehat{\frac{\partial}{\partial x^j}}\right)&=&\nonumber X^{c}(\tilde g_{lj})-\tilde g\left(\mathcal L_X\widehat{\frac{\partial}{\partial x^l}},\widehat{\frac{\partial}{\partial x^j}}\right)-\tilde g\left(\widehat{\frac{\partial}{\partial x^l}},\mathcal L_{X}\widehat{\frac{\partial}{\partial x^j}}\right)\\&=&X^{c}(\tilde g_{lj})+\tilde g\left(\frac{\partial X^h}{\partial x^l}\widehat{\frac{\partial}{\partial x^h}},\widehat{\frac{\partial}{\partial x^j}}\right)+\tilde g\left(\widehat{\frac{\partial}{\partial x^l}},\frac{\partial X^h}{\partial x^j}\widehat{\frac{\partial}{\partial x^h}}\right)\nonumber\\&=& X^{c}(\tilde g_{lj})+\frac{\partial X^h}{\partial x^l}\tilde g_{hj}+\frac{\partial X^h}{\partial x^j}\tilde g_{lh},
\end{eqnarray}
where $\tilde g_{ij}:=\tilde g\left(\widehat{\frac{\partial }{\partial x^i}},\widehat{\frac{\partial }{\partial x^j}}\right)$, for all $i,j\in \{0,\ldots,n\}$; here, in the second equality, we have used (\ref{vLie}) and the fact that  $i(\widehat{\frac{\partial }{\partial x^i}})=\frac{\partial}{\partial y^i}$. 
\begin{dfn}
Let $(\tilde M, L, A)$ be a Finsler spacetime, $K$ be a smooth vector field on $\tilde M$ and $\psi$ its flow. We say that $K$ is a {\em Killing vector field of $(\tilde M,L,A)$} if $\mathcal L_K\tilde g=0$. 
\end{dfn}
The following characterization of Killing vector fields holds:
\begin{prop}\label{Linvariant}
	Let $(\tilde M, L, A)$ be a Finsler spacetime (hence $L\in C^0(T\tilde M)\cap C^3(A)$, according to Definition~\ref{fst}), then   $K$ is a Killing vector field, if and only if  $K^c(L)|_A=0$. 
\end{prop}
\begin{proof}
Observe that from \eqref{completelift} we have 
\begin{eqnarray*}
K^{c}(L)(x,y)&=&K^{h}(x)\frac{\partial L}{\partial x^{h}}(x,y)+\frac{\partial K^{h}}{\partial x^{i}}(x)y^{i}\frac{\partial L}{\partial y^{h}}(x,y)\\&=&K^{h}(x)\frac{\partial }{\partial x^{h}}\big(\tilde g_{lj}(x,y)y^l y^j\big)+\frac{\partial K^{h}}{\partial x^{i}}(x)y^{i}\frac{\partial}{\partial y^{h}}\big(\tilde g_{lj}(x,y)y^l y^j\big)\\&=&\left(K^{c}(\tilde g_{lj})(x,y)+\frac{\partial K^h}{\partial x^l}(x)\tilde g_{hj}(x,y)+\frac{\partial K^h}{\partial x^j}(x)\tilde g_{lh}(x,y)
\right)y^{l}y^j
\end{eqnarray*}
for every $(x,y)\in A$.
Thus, if  $K$ is Killing, by (\ref{jjj}),  $K^{c}(L)(x,y)= \big(\mathcal L_{K}\tilde g\big)_{(x,y)}(y,y)=0$, for every $(x,y)\in A$. 

Let us now assume that $K^c(L)|_A=0$. Observe that if there exists an open subset $U$ of $\tilde M$ where $K$ vanishes then also $K^c|_U\equiv 0$ and, from \eqref{jjj}, $(\mathcal L_K\tilde g)_{(x,y)}=0$ for all $(x,y)\in TU\cap A$.    Let, then, $p\in \tilde M$ such that $K_p\neq 0$ and let us consider, in a neighborhood $V\subset \tilde M$ of $p$, a coordinate system $(x^0,\ldots, x^n)$ such that $\frac{\partial}{\partial x^0}=K$. In the induced  coordinates on $T\tilde M$  we have that  $K^c=\frac{\partial}{\partial x^0}\circ\tilde{\pi}$  (recall \eqref{completelift}) and, so, $K^c(L)(x,y)=\frac{\partial L}{\partial x^0}(x,y)=0$ for all $(x,y)\in TV\cap A$. Hence,
\[\frac{\partial^3 L}{\partial y^i\partial y^j \partial x^0}(x,y)=\frac{\partial^3 L}{\partial x^0\partial y^i\partial y^j }(x,y)=2\frac{\partial\tilde g_{ij}}{\partial x^0}(x,y)=0,\] 
for all $(x,y)\in TV\cap A$ and for all $i,j\in \{0,\ldots,n\}$; from \eqref{jjj}, $(\mathcal L_K\tilde g)_{(x,y)}=0$ for each $(x,y)\in TV\cap A$. Thus, $(\mathcal L_K\tilde g)_{(p,y)}=0$ for any  $(p,y)\in A$ such that $K_p\neq 0$. By continuity, $(\mathcal L_K\tilde g)_{(q,y)}=0$ for any $(q,y)\in A$ such that $q$ belongs to the closure of $\{p\in \tilde M:K_p\neq 0\}$ and then $\mathcal L_K\tilde g=0$ everywhere in $A$. 
 \end{proof}
Let us now see  how the flow of $K^c$ behaves w.r.t. $A$.   
\begin{lem}\label{Ainvariant}
Let $\tilde M$  be a manifold and $A$ a cone subset of $T\tilde M$ (according to the definition in the Introduction). Let $X$ be a smooth vector field on $\tilde M$. Then for each $\bar p\in \tilde M$ there exists an interval  $I_{\bar p}$, $0\in I_{\bar p}$, and a neighborhood $U$ of $\bar p$ in $\tilde M$ such that the flow $\tilde \psi$ of $X^c$ is well defined on $I_{\bar p}\times TU$ and $\tilde \psi\big(I_{\bar p}\times (TU\cap A)\big)\subset A$.   
\end{lem}
\begin{proof}
Let us denote by $\psi$ the flow of $X$. It is well known that for any $v\in T\tilde M$ there exists a neighborhood $V \subset \R\times T\tilde M$ of $(0,v)$ such that 
\[\tilde \psi	: V\to T\tilde M,\quad \quad \tilde\psi(t,v)=(\psi(t,p),d\psi_t(v)),\]
$p=\tilde \pi(v)$,
is a local  flow of $X^c$. In fact, 
\[\frac{\partial \tilde \psi}{\partial t}(t,v)=\left(\frac{\partial  \psi}{\partial t}(t, p),\frac{\partial  }{\partial t}\big(\partial_x \psi(t,p)(v)\big)\right),\]
where $\partial_x \psi(t,p)(v)$ denotes the partial differential of $\psi$ w.r.t. the second variable at $(t,p)$, evaluated in $v$ (hence $\partial_x \psi(t,p)(v)=d \psi_t(v)$). Thus, in local coordinates on $T\tilde M$, we have
\begin{align*}
\frac{\partial \tilde \psi}{\partial t}(t,v)&= X^h\big(\psi(t,p)\big)\frac{\partial}{\partial x^h}+\frac{\partial  }{\partial t}\left(\frac{\partial \psi^h}{\partial x^j}(t,p) v^j\right)\frac{\partial}{\partial y^h}\\
&=X^h\big(\psi(t,p)\big)\frac{\partial}{\partial x^h}+\frac{\partial^2  \psi^h}{\partial t\partial x^j}(t,p)v^j\frac{\partial}{\partial y^h}\\
&=X^h\big(\psi(t,p)\big)\frac{\partial}{\partial x^h}+\frac{\partial }{ \partial x^j}\Big(X^h\big(\psi(t,p)\big)\Big)v^j\frac{\partial}{\partial y^h}\\
&=X^h\big(\psi(t,p)\big)\frac{\partial}{\partial x^h}+\frac{\partial X^h}{\partial x^l}\big(\psi(t,p)\big)\frac{\partial \psi^l}{\partial x^j}(t,p)v^j\frac{\partial}{\partial y^h}\\
&=X^c\big(\tilde \psi(t,p)\big).
\end{align*}
As $\frac{\partial \psi^l}{\partial x^j}(0,p)=\delta^l_j$ for all $p\in M$, where $\delta^l_j$ are the Kronecker symbols,   and $\psi$ is smooth, we have that for any $\bar p\in \tilde M$  any $\epsilon>0$ there exists an interval $I_{\bar p}$, centered at $0$, and a neighborhood  $U$ of $\bar p$ in $\tilde M$ such that $\psi$ is well defined in $I_{\bar p}\times U$ and  
\[ \left|\frac{\partial \psi^l}{\partial x^j}(t,p)-\delta^l_j\right|<\epsilon,\]
for all $(t,p)\in I_{\bar p}\times U$ and each $l,j\in\{0,\ldots, n\}$. Hence, for any $u=(u^0,\ldots, u^n)\in \R^n$ such that $|u|=1$ we have 
\[ \left|\frac{\partial \psi^l}{\partial x^j}(t,p)u^j-\delta^l_ju^j\right|<(n+1)\epsilon,\]
for each $l\in\{0,\ldots, n\}$.
 Being $A$ open and $A_p$ a cone for all $p\in \tilde M$, we then conclude that the vector $\left(\frac{\partial \psi^0}{\partial x^j}(t,p)v^j, \ldots, \frac{\partial \psi^n}{\partial x^j}(t,p)v^j\right)\in A_{\psi(t,p)}$, for all $(t,p)\in I_{\bar p}\times U$, provided that $(v^0,\ldots, v^n)\in A_p$. Thus  the flow of $X^c$ is well defined on $I_{\bar p}\times TU$ and $\tilde \psi \big(I_{\bar p}\times (TU\cap A)\big)\subset A$.   
\end{proof}
From Propositions~\ref{Linvariant} and \ref{Ainvariant} it follows that $L$ is invariant under the flow of $K^c$. In fact,  $0=\big(K^c(L)\big)(\tilde \psi_s(v))=\frac{d}{dt}L(\tilde \psi_{s+t} (v))|_{t=0}$, for all $v\in A$ and $s\in I_{\pi(v)}$, hence $s\in I_{\pi(v)}\mapsto L(\tilde \psi_s(v))$ is constant. From this observation 
we get that Killing vector fields are also the infinitesimal generators of local $\tilde g$-isometries:
\begin{prop}
	Let $(\tilde M, L, A)$ be a Finsler spacetime, $K$ be a smooth vector field on $\tilde M$ and let us denote by $\psi$ the flow of $K$. Then $K$ is a Killing vector field if and only if  
for each $v\in A$ and for all $v_1,v_2\in  T_{\pi(v)}\tilde M$, we have  
\begin{equation}\label{isometry}
\tilde g_{d\psi_{t}(v)}\big(d \psi_{ t}(v_1),d \psi_{t}(v_2)\big)= \tilde g_{v}(v_1,v_2),
\end{equation}
for all $t\in I_p$, where $I_p\subset\R$ is an interval containing $0$  such that   the stages  $\psi_{t}$  are well defined in a neighbourhood $U\subset \tilde M$ of $p=\pi(v)$ and $d\psi_t(v)\in A$, for each $t\in I_p$. \end{prop}
\begin{proof}
Let $v\in A$ and $p=\pi(v)$. From Lemma~\ref{Ainvariant}, the flow $\tilde \psi$ of $K^c$ is well defined in $I_p\times TU$, for an interval $I_p$ containing $0$, ($\tilde \psi(0,v)=v$), and a neighborhood $U$ of $p$ in $\tilde M$ and, moreover, $\tilde \psi\big(I_p\times(TU\cap A)\big)\subset A$.
If \eqref{isometry} holds, then in particular $L(d\psi_t(v))=g_{d\psi_{t}(v)}\big(d \psi_{ t}(v),d \psi_{t}(v)\big)=g_v(v,v)=L(v)$, for all $t\in I_p$. Hence $0=\frac{d}{dt}L(d\psi_t (v))|_{t=0}=\big(K^c(L)\big)(v)$ and then we conclude using Proposition~\ref{Linvariant}. The converse  follows  observing that, being $L$ invariant under the flow of $K^c$, 
\baln
\lefteqn{g_{d\psi_{t}(v)}\big(d \psi_{ t}(v_1),d \psi_{t}(v_2)\big)}&\\
&=\frac 1 2 \frac{ \partial^{2}L}{\partial s_1\partial s_2}\big(d\psi_{t}(v)+s_1d\psi_{t}(v_1)+s_2d\psi_{t}(v_2)\big)|_{(s_1,s_2)=(0,0)}\\&=
\frac 1 2 \frac{ \partial^{2}L}{\partial s_1\partial s_2}\big(d\psi_{t}(v+s_1v_1+s_2v_2)\big)|_{(s_1,s_2)=(0,0)}\\
&=\frac 1 2 \frac{ \partial^{2}L}{\partial s_1\partial s_2}(v+s_1v_1+s_2v_2)|_{(s_1,s_2)=(0,0)}=g_v(v_1,v_2).
\ealn
\end{proof}

\section{Stationary splitting Finsler spacetimes}
As in the Lorentzian setting, we say that a Finsler spacetime $(\tilde M,L, A)$ is \emph{stationary} if it admits a timelike Killing vector field. Here {\em timelike} means that  $L(K_p)<0$ for all $p\in\tilde M$.\footnote{\label{formal} Analogously a vector $v\in T\tilde M$ is said lightlike (resp. spacelike; causal ) if $L(v)=0$ (resp. either $L(v)>0$ or $v=0$; $L(v)\leq 0$); anyway observe that,  being $\tilde g$ defined only on $A$, $L(v)=\tilde g_v(v,v)$ only for vectors $v\in A$, so whenever $v\not \in \bar A$ this causal character is purely formal, and it is  no way  related to the generalised metric $g$.}

A particular type of stationary Lorentzian manifolds (called {\em standard stationary}) can be obtained starting from a product manifold $\tilde M=\R\times M$, a Riemannian metric $g$, a one-form $\omega$ and a positive function $\Lambda$ on $M$, by considering the Lorentzian metric:
\begin{equation}
\tilde g=-\Lambda  dt^2+\omega\otimes d t+d t\otimes\omega+g.\label{sss}
\end{equation}
It is well known (see e.g. \cite[Appendix C]{GiaPic99}) that any stationary Lorentzian spacetime is locally isometric to a standard one. 

Looking at the quadratic form associated to the  Lorentzian metric \eqref{sss} with the aim of introducing a Finslerian analogue, we are led to  the Lagrangian $L:T\tilde M\to \R$,
\begin{equation}\label{stationary}
L(\tau,v)=-\Lambda\tau^2+2B(v)\tau+F^2 (v),
\end{equation}
where $\Lambda$ is a positive function  on $M$, $B\colon TM\to \R$ is  a fiberwise positively homogeneous of degree $1$ Lagrangian  which is at least $C^3$ on $TM\setminus 0$ and $F\colon TM\to [0,+\infty)$, $F\in C^0(TM)\cap C^3(TM\setminus 0)$ is a classical Finsler metric on $M$, i.e, it is fiberwise positively homogeneous of degree $1$ and 
\[g_v(u,u):=\frac 1 2\frac{\partial^2 F^2}{\partial s_1\partial s_2} (v+s_1u+s_2u)|_{(s_1,s_2)=(0,0)}>0\]
for all $v\in TM\setminus 0$ and all $u\in T_{\pi_M(v)}M$, where $\pi_M$ is the canonical projection $\pi_M:TM \to M$.  

Let us now introduce some notation. 
We will denote coordinates $(t,x^1, \ldots, x^n)$, in $\tilde M=\R\times M$ by   $z$, i.e.  $z=(t,x^1, \ldots, x^n)$. Natural coordinates in $T\tilde M$,  
will be then denoted by $(z,\dot z)$, that is  $(z,\dot z)=(t,x^1, \ldots, x^n,\tau, y^1,\ldots, y^n)$.
For a Lagrangian $A:TM\to\R$ and $v\in TM\setminus 0$, let us denote by $(\partial_yA)_v$ and $(\partial^2_{yy}A)_v$, respectively, the fiberwise differential and Hessian of $A$ at $v$, i.e. for all $u,u_1,u_2\in TM$
\baln
&(\partial_yA)_v(u):=\frac{d}{ ds}\big(A(v+su)\big)|_{s=0},\\ &(\partial^2_{yy}A)_v(u_1,u_2):=    \frac{\partial^2 }{\partial s_1\partial s_2}\big(A(v+s_1u_1+s_2u_2)\big)|_{(s_1,s_2)=(0,0)}.\ealn
These are respectively  sections of the pull-back bundles $\pi^{*}_M(T^{*}M)$ and $\pi^{*}_{M}(T^{*}M)\otimes\pi^{*}_{M}(T^{*}M)$ over $TM\setminus 0$.

The analogous fiberwise  derivatives, for a Lagrangian  $L\colon T\tilde M\to\R$ on $\tilde M$, are denoted  by $(\partial_{\dot z}L)_{w}$ and $(\partial^2_{\dot z\dot z}L)_{w}$, $w\in T\tilde M$,
and  when $L$ is a Lorentz-Finsler function on $\tilde M$ then $\frac 12 (\partial^2_{\dot z\dot z}L)_{w}$ is, then, the generalised metric  tensor $\tilde g_w$ already introduced in \eqref{tildeg}.

Let us denote by $\mathcal T$ the trivial line subbundle of $T\tilde M$ defined by the vector field $\partial_t$. Let $z=(t,x)\in \tilde M$ and let us denote  by $T^+_p\tilde M$ and $T^-_p\tilde M$ respectively the open half-spaces  of $T_p\tilde M$ given by $T^+_p\tilde M:=\{(\tau, v)\in T_p\tilde M: \tau >0\}$ and $T^-_p\tilde M:=\{(\tau, v)\in T_p\tilde M: \tau <0\}$; moreover let $\bar T^+_p\tilde M$ and $\bar T^-_p\tilde M$ be their closures in $T_p\tilde M$.  Let us then denote by
$T^+\tilde M\setminus \mathcal T$ (resp. $T^-\tilde M\setminus \mathcal T$) the open cone subset  of $T\tilde M$ given by  $T^+\tilde M\setminus \mathcal T:=\cup_{p\in \tilde M}T^+_p\tilde M\setminus \mathcal T_p$ (resp. $T^-\tilde M\setminus \mathcal T:=\cup_{p\in \tilde M}T^-_p\tilde M\setminus \mathcal T_p$) and by 
$\bar T^+\tilde M\setminus \mathcal T$ (resp. $\bar T^-\tilde M\setminus \mathcal T$) the cone subset defined by $\bar T^+\tilde M\setminus \mathcal T:=\cup_{p\in \tilde M}\bar T^+_p\tilde M\setminus \mathcal T_p$ (resp. $\bar T^-\tilde M\setminus \mathcal T:=\cup_{p\in \tilde M}\bar T^-_p\tilde M\setminus \mathcal T_p$). Finally, let us denote by $T\tilde M\setminus \mathcal T$ the open cone subset of $T\tilde M$ defined as $T\tilde M\setminus \mathcal T:=\cup_{p\in \tilde M}T_p\tilde M\setminus \mathcal T_p$.

\bere Notice that $L$ is continuous on $T\tilde M$ and at least $C^3$ on $T\tilde M\setminus\mathcal T$.  Since, in general, $B$ is not  differentiable at the zero section of $TM$, $L$ is  not differentiable at vectors $y\in \mathcal T$. An exception is when $B$ reduces to a one-form on $M$ (being, so, differentiable at $0$ too)  so that $L$ is $C^1$ on $T\tilde M\setminus 0$:
\ere

\begin{prop}\label{linearity}
Let $L\colon T\tilde M\to \R$ defined as in \eqref{stationary} and $w\in \mathcal T_{(t,x)}$, $(t,x)\in\R\times M$. Then $L$  admits fiberwise derivative $(\partial_{\dot z} L)_{w}$   if and only if the map $B_x:T_xM\to \R$, $B_x(v):=B(v)$, is odd. Moreover, in this case, $(\partial_{\dot z} L)_{w}$ is a linear map on $T_{(t,x)}\tilde M\equiv\R\times T_xM$ if and only if $B_x$ is linear.
\end{prop}
\begin{proof}
Let $w=(\tau,0)\in \R\times T_xM$ and let us compute $(\partial_{\dot z} L)_{w}$.   For all $u\equiv(\rho,v)\in \R\times T_xM$ we have
\[
L(w+su)=-\Lambda(x)(\tau+s\rho)^2+2B_x(sv)(\tau+s\rho)+F^2(sv), 
\]
hence the right and left derivative at $s=0$ of $L(w+su)$ do exist and are respectively equal to $-2\Lambda(x)\tau \rho +2B_x(v)\tau$ and  $-2\Lambda(x)\tau \rho -2B_x(-v)\tau$;
thus  they are equal if and only if $B_x(v)=-B_x(-v)$. In this case, being $B_x$ positively homogeneous of degree $1$, $(\partial_{\dot z} L)_{w}(u)=-2\Lambda(x)\tau \rho +2B_x(v)\tau$ is linear in $u\equiv(\rho,v)$ if and only if $B_x$ is linear.
\end{proof}
We  characterize now   when, for   $L$  in \eqref{stationary}, $(\partial^2_{\dot z\dot z}L)_{w}$ has index $1$ for all $w\in T^{\pm}\tilde M\setminus\mathcal T$,  and for all $w\in T\tilde M\setminus\mathcal T$.

\begin{prop}\label{index1}
Let $L\colon T\tilde M\to \R$ defined as in \eqref{stationary}, then $(\partial^2_{\dot z\dot z}L)_{w}$ has index $1$  at   $w\in \bar T^+\tilde M\setminus \mathcal T$ (resp.  $w\in \bar T^-\tilde M\setminus \mathcal T$), $w=(\tau, v)$,
if  $(\partial^2_{yy}B)_v$ is positive semi-definite (resp.  $(\partial^2_{yy}B)_v$ is negative semi-definite). Conversely, if there exists $\bar\tau>0$ (resp. $\bar \tau <0$) such that $(\partial^2_{\dot z\dot z}L)_{(\tau,v)}$ has index $1$ for $v\in TM\setminus 0$ and  for all $\tau>\bar \tau$ (resp. $\tau<\bar \tau $) then $(\partial^2_{yy}B)_v$ is positive semi-definite (resp.  $(\partial^2_{yy}B)_v$ is negative semi-definite). 
\end{prop}
\begin{proof}
Let us prove the sufficient condition in the first equivalence. For  $(\tau,v)\in \bar T^+\tilde M\setminus \mathcal T$, the fiberwise Hessian of $L$  is given by
\begin{equation}\label{tildegstat}
(\partial^2_{\dot z\dot z}L)_{(\tau, v)}=-\Lambda d t^2 +(\partial_y B)_v \otimes d t +d t\otimes (\partial_y B)_v+\tau (\partial ^2_{yy}B)_v+g_v
\end{equation}
First notice that, if $(\partial_y B)_v=0$ then, 
being $\Lambda$  positive and $\tau$ non-negative,  we  immediately get that $(\partial^2_{\dot z\dot z}L)_{(\tau, v)}$ has index $1$. Assume now that $(\partial_y B)_v\neq 0$ and consider the vector $v_1\in T_{\pi_M(v)}M$ representing $(\partial_y B)_v$ w.r.t. the scalar product $\langle \cdot, \cdot\rangle_{v,\tau}$ defined by  $g_v+\tau(\partial^2_{yy}B)_v$. Now take a $(\langle \cdot, \cdot\rangle_{v,\tau})$-orthonormal basis $v_2,...,v_{n}$ of $\mathrm{Ker}((\partial_y B)_v)$.  The matrix of $ (\partial^2_{\dot z\dot z}L)_{(\tau, v)}$ relative to the basis $e,v_1,v_2,...,v_{n}$ of $T_{\pi((\tau,v))}\tilde M$, where $e=\frac{\partial}{\partial t}$ is given by
\[\left(\begin{array}{ccccc}
-\Lambda&(\partial_yB)_v(v_1)&0&\ldots&0\\ (\partial_yB)_v(v_1)&\langle v_1, v_1\rangle_{v,\tau}&0&\ldots&0\\ 0&0&&&\\
\vdots&\vdots&&I&\vspace{6pt}\\
0&0&&& 
\end{array}
\right)
\]
Since $-\Lambda \langle v_1, v_1\rangle_{v,\tau}-\big((\partial_yB)_v(v_1)\big)^2
<0$, we conclude that $(\partial^2_{\dot z\dot z}L)_{(\tau, v)}$ has index $1$. Conversely, let us assume that  $ (\partial^2_{\dot z\dot z}L)_{(\tau, v)}$ has index $1$, for any $(\tau,v)\in T^+\tilde M\setminus \mathcal T$ with $\tau>\bar{\tau}$. For a fixed $(\tau,v)\in T^+\tilde M\setminus \mathcal T$, with $\tau>\bar \tau$, consider the Lorentzian metric $\tilde g_{(\tau,v)}:= (\partial^2_{\dot z\dot z}L)_{(\tau, v)}$ and the $g_{(\tau,v)}$-orthogonal complement, $\mathcal D_{\pi((\tau,v))}$, in $T_{\pi((\tau,v))}\tilde M$ of the one-dimensional subspace generated by the vector $(1,0)$. This is given by vectors $(\tau_1,v_1)$ with $\tau_1=\frac{(\partial_{\dot z}B)_v(v_1)}{\Lambda}$, for each $v_1\in TM$. Thus
\begin{multline}
\tilde g_{(\tau,v)}\left(\Big(\frac{(\partial_{\dot z}B)_v(v_1)}{\Lambda}, v_1\Big), \Big(\frac{(\partial_{\dot z}B)_v(v_1)}{\Lambda}, v_1\Big)\right)\\=-\frac{\big((\partial_{\dot z}B)_v(v_1)\big)^2}{\Lambda}+\tau (\partial ^2_{yy}B)_v(v_1,v_1)+g_v(v_1,v_1)>0,\label{onorto}
\end{multline}
for all $v_1\in TM\setminus 0$. If there exists $v_1\in TM\setminus 0$ such that $(\partial ^2_{yy}B)_v(v_1,v_1)<0$ then, being $\mathcal D_{\pi((\tau,v))}$ independent of $\tau$, for any fixed $v\in TM$, we can consider the limit as $\tau\to+\infty$ in the left-hand  side of \eqref{onorto}, obtaining a negative quantity for $\tau$ big enough, in contradiction with \eqref{onorto}. 

It is easy to check that the analogous statement involving    $(\partial ^2_{yy}B)_v$ negative semi-definite, holds as the previous one with obvious modifications. 
\end{proof}
\begin{cor}\label{oneform}
Let $L\colon T\tilde M\to \R$ defined as in \eqref{stationary} and $x\in M$. Then  $(\partial^2_{\dot z\dot z}L)_{(\tau,v)}$ has index $1$  for all $v\in T_xM\setminus\{0\}$ and all $\tau \in \R$  if and only if $B(x,\cdot)$ is linear form on $T_xM$.    
\end{cor}
\begin{proof}
It is trivial to check that if $B(x,\cdot)$ is a linear form on $T_x M$ then $(\partial^2_{\dot z\dot z}L)_{(\tau,v)}$ has index $1$ for all $\tau\in\R$ and all $v\in  T_x M\setminus \{0\}$. Conversely, from Proposition~\ref{index1}  $(\partial ^2_{yy}B)_v$ must vanish on $T_xM\setminus 0$ and, being $B$ fiberwise positively homogeneous of degree $1$,  it must be linear  on $T_xM$.
\end{proof}
As the fiberwise Hessian of a classical Finsler metric  is positive semi-definite, from Proposition~\ref{index1} we immediately get:
\begin{cor}
Let $L\colon T\tilde M\to \R$ defined as in \eqref{stationary} with   $B=\omega+F_1$ (resp. $B=\omega-F_1$), where $\omega$ and $F_1$ are, respectively, a one-form and a Finsler metric on $M$.  Then $(\partial^2_{\dot z\dot z}L)_{w}$  has index $1$ for all $w\in  T^+\tilde M\setminus \mathcal T$ (resp. $w\in  T^-\tilde M\setminus \mathcal T$).
\end{cor}
\bere \label{whole}
We have found that  if the  conditions of  Proposition~\ref{index1} and Corollary~\ref{oneform} hold on the whole $T^+\tilde M\setminus \mathcal T$ (resp. $T^-\tilde M\setminus \mathcal T$; $T\tilde M\setminus \mathcal T$) then $(\mathbb R\times M,L, T^+\tilde M\setminus \mathcal T)$ (resp. $(\mathbb R\times M,L, T^-\tilde M\setminus \mathcal T)$; $(\mathbb R\times M, L, T\tilde M\setminus \mathcal T)$) is a Finsler spacetime. 
Notice, indeed, that  the role of the vector field $Y$, such that $Y_p\in \bar A_p$ for all $p\in \mathbb R\times M$, can be taken by $\partial_t$ in the first and in the last case and by $-\partial_t$ in the second one. 
\ere
\bere
We could  consider the case when the  assumptions on $B$ hold  pointwise  for all  $x\in M$ being all the three cases possible. Anyway, take into account that we would not get a Finsler spacetime due to the impossibility of fulfilling the assumption about the existence of a smooth vector field $Y$ such that  $Y_p\in  \bar{A_p}$ and $L(p,Y_p)<0$, for all $p\in \tilde M$. On the other hand,  the case when $(\partial_{yy} B)_v$ is either positive or negative semi-definite for all $v\in TM$ include also the possibility that $B(x,\cdot)$ is a linear form on $T_xM$ for some $x\in M$.
\ere

Henceforth, we will  denote  by $(\tilde M, L)$,  $\tilde M=\mathbb R\times M$, each of the Finsler spacetimes $(\tilde M,L, T^+\tilde M\setminus \mathcal T)$, $(\tilde M,L, T^-\tilde M\setminus \mathcal T)$, $(\tilde M, L, T\tilde M\setminus \mathcal T)$,  associated to $L$ given in \eqref{stationary},  implicitly assuming that if $\partial ^2_{yy}B$ is positive semi-definite (resp. $\partial ^2_{yy}B$ is negative semi-definite; $B$ is a one-form on $M$) then $\tilde g$ is defined and has index $1$  on the cone subset $A$ given by $T^+\tilde M\setminus \mathcal T$ (resp. $T^-\tilde M\setminus \mathcal T$; $T\tilde M\setminus \mathcal T$).

\bere\label{futurepastpoint}
In analogy with the Lorentzian case,  we say that a causal vector $w$ (recall footnote~\ref{formal}), with $w\in \bar A\setminus 0$ is {\em future-pointing} (resp. {\em past-pointing}) if $g_w(w,Y)< 0$ (resp. $g_w(w,Y)> 0$), whenever $w\in A$, or $w$ is causal and belongs to the closure of the set of future-pointing vectors in $A$.   In the case of a stationary splitting Finsler spacetime, taking into account that  when $A=T^-\tilde M\setminus \mathcal T$ we pick  $-\partial_t$ as the vector field $Y$,  we have that a causal vector  $(\tau,v)$ of $(\tilde M, L, T^+\tilde M\setminus \mathcal T)$ (resp. of $(\tilde M, L, T^-\tilde M\setminus \mathcal T)$; $(\tilde M, L, T\tilde M\setminus \mathcal T)$)
with $(\tau, v)\in T^+\tilde M\setminus \mathcal T$ (resp. $(\tau, v)\in T^-\tilde M\setminus \mathcal T$; $(\tau, v)\in T\tilde M\setminus \mathcal T$ ) is {\em future-pointing}  if $-\Lambda \tau +(\partial_y B)_v (v)<0$ (resp. $-\Lambda \tau +(\partial_y B)_v (v)>0$; $-\Lambda \tau +B(v)<0 $). By homogeneity, the first (resp. second) inequality becomes  $-\Lambda \tau +B(v)<0$  (resp. $-\Lambda \tau +B(v)>0$). 
Being $B(0)=0$, the vectors of the type $(\tau,0)$, $\tau>0$, (resp. $\tau<0$) are then also timelike and future-pointing (resp. past-pointing). We will see in Remark~\ref{futureinside} that the  causal future-pointing causal vectors of $(\tilde M, L, T^-\tilde M\setminus \mathcal T)$ (resp. $(\tilde M, L, T^-\tilde M\setminus \mathcal T)$)  at $p\in\tilde M$ are all and only the causal vectors in $\bar T^+\tilde M:=\cup_{p\in \tilde M}\bar T^+_p\tilde M$   (resp. $\bar T^-\tilde M:=\cup_{p\in \tilde M}\bar T^-_p\tilde M$).
%
\ere
\begin{prop}\label{desudetkilling}
Assume that $(\tilde M, L)$, $\tilde M=\mathbb R\times M$ and $L$  as in \eqref{stationary}, is a Finsler spacetime, then  $\partial_t$ is timelike and Killing. 
\end{prop}
\begin{proof}
We consider the fundamental tensor $\tilde g$ of $L$ in \eqref{tildegstat}. Let us prove that $\mathcal L_{\partial_t}\tilde g(\widehat{\partial_{z^i}}, \widehat{\partial_{z^j}})=0$ for all $i,j\in\{0, \ldots, n\}$, where $z^0=t$ and $z^i=x^i$ for all $i\in\{1,\ldots,n\}$. From \eqref{jjj}, it is enough to prove that $(\partial_t)^c (\tilde g_{ij})=0$, for all $i,j\in\{0, \ldots, n\}$. 
From \eqref{completelift}, we have
$(\partial_t)^c (\tilde g_{ij})=\partial_t\tilde g_{ij}=0$.   
\end{proof}
 
Proposition~\ref{index1}, Corollary~\ref{oneform}, Remark~\ref{whole} and Proposition~\ref{desudetkilling} justify the following definition:
\begin{dfn}
Let $\tilde M=\R\times M$,  $L\colon T\tilde M\to \R$ defined as in \eqref{stationary}, such that $\partial^2_{yy} B$ is positive semi-definite on $TM\setminus 0$ (resp. $\partial^2_{yy} B$ is negative semi-definite on $TM\setminus 0$)  then we call $(\tilde M, L)$  a \emph{stationary splitting Finsler spacetime}.
\end{dfn}

\section{On the local structure of stationary Finsler spacetimes}
An important role in several  geometric  properties of stationary Lorentzian manifolds is played by the  distribution $\mathcal D$ orthogonal to the Killing vector field $K$. For example, since  $K$ is timelike, $\mathcal D$ is  spacelike and then, in a standard stationary Lorentzian manifold $(\R\times M, \tilde g)$, $\tilde g$ given in   \eqref{sss}, it is the horizontal distribution of the semi-Riemannian submersion $\pi\colon \R\times M\to M$, where the Riemannian  metric on $M$ is equal to $g+\frac{\omega}{\Lambda}\otimes \omega$ (see, e.g.,  \cite{CaJaPi10}). Moreover, if  $\mathcal D$ is integrable  then a stationary Lorentzian manifold $(\tilde M,\tilde g)$ is said {\em  static} (see \cite[Def. 12.35]{O'Neil}) and  it is locally isometric  to a warped product $(a,b)\times S$, with the metric $-\Lambda d t^2 + g$, where $S$ is an integral manifold of the distribution, $(a,b)\subset \R$, $\phi_*(\partial_t)=K$,  being  $\phi\colon (a,b)\times S\to M$ a local isometry, $g(\phi^{*}u,\phi^{*}v)=\tilde g(u,v)$, for all $u,v\in \mathcal D$
(see \cite[Prop. 12.38]{O'Neil}).

A natural  generalisation of the orthogonal distribution to $K$ in the Finsler setting is the distribution in $T\tilde M$ defined as  $\ker(\partial_{\dot z} L(K))$  where $\partial_{\dot z}L(K)$ denotes the one-form on $\tilde M$ given by $\frac{\partial L}{\partial \dot z^i}(K)dz^i$.
\bere \label{gooddistro}
In order to get a well defined distribution  $\mathrm{Ker}\big(\partial_{\dot z} L(K)\big)$, we need that  $L$ is differentiable at $K_z$ for all $z\in \tilde M$. Thus, we assume  that $L$ is $C^1$ on $T\tilde M$ whenever we need to consider such a distribution, as in Theorem~\ref{charstat}. Recall that for a stationary splitting Finsler spacetime $(\R\times M, L)$, this assumption implies that $B$ reduces to a one-form on $M$ (Proposition~\ref{linearity}) that we will denote, in this section,  by $\omega$. Recall  that  from Corollary~\ref{oneform},  $\tilde g$ is then defined on  $T\tilde M\setminus \mathcal T$. 
\ere
Following \cite{CapSta16}, we  introduce the next two  definitions:
\begin{dfn}
	We say that a Finsler spacetime $(\tilde M, L, A)$ is {\em static} if there exists a timelike  Killing vector field $K$  such that the distribution of hyperplanes  $\ker(\partial_{\dot z} L(K))$ is integrable.
\end{dfn}
\begin{dfn}
	We say that a  Finsler spacetime $(\tilde M, L, T\tilde M\setminus\mathcal T)$, where $\mathcal T$ is a line subbundle of $T\tilde M$,  is {\em standard static} if there exist  a smooth non vanishing global section $K$ of $\mathcal T$, a Finsler manifold $(M,F)$, a positive function $\Lambda$ on $M$  and a  smooth diffeomorphism $f\colon \R\times M\to \tilde M$, $f=f(t,x)$, such that  $\partial_t=f^*(K)$ and $L(f_*(\tau, v))=-\Lambda \tau^2+ F^2(v)$, for all $(\tau, v)\in T(\R\times M)$.
\end{dfn}

In relation to the local structure of a stationary Finsler spacetime, we introduce also the following definition:
\begin{dfn}\label{LSS}
A stationary Finsler spacetime $(\tilde M,L, T\tilde M\setminus\mathcal T )$, where $\mathcal T$ is a line subbundle of $\tilde M$,  with timelike Killing field $K$, $K_z\in\mathcal T_z$, for all $z\in \tilde M$, is  {\em locally a standard stationary splitting}  if for any point $z \in \tilde M$ there exists a neighborhood $U_z\subset \tilde M$ of $z$ and a diffeomorphism $\phi:I_{z}\times S_z\rightarrow U_z$, where $I_z=(-\varepsilon_z,\varepsilon_z)$ is an interval in $\mathbb{R}$ and $S_z$ a manifold, such that, named $t$ the natural coordinate of $I_{z}$, $\phi_{*}(\partial_t)=K|_{U_z}$,  
and for all $(\tau,v)\in T(I_{z}\times S_{z})$, $L\circ\phi_*((\tau,v))=-\Lambda\tau^2+\omega(v)\tau+F^{2}(v)$, where $\Lambda$, $\omega$ and $F$ are respectively a positive function, a one-form and a Finsler metric on $S_z$. Moreover, we say that $(\tilde M,L)$ is {\em locally standard static} if for any $z\in \tilde M$ there exists a map $\phi$ as above such that $\omega =0$.

\end{dfn}
\begin{rmk}
We observe that, although $L$ might be not twice differentiable along vectors $w\in\mathcal{T}$, its fiberwise second derivative at $w\in \mathcal T\setminus 0$, evaluated at any couple of vectors $u_1,u_2\in \mathcal T_{\pi(w)}$,  $(\partial^{2}_{\dot z\dot z} L)_{w}(u_1,u_2):=\frac{\partial^2}{\partial s_1\partial s_2}L(w+s_1u_1+s_2u_2)|_{(s_1,s_2)=(0,0)}$, does exist. Indeed, let $\lambda_1,\lambda_2\in\R$ such that $u_1=\lambda_1 w$, $u_2=\lambda_2w$, then by homogeneity, we have 
\begin{equation*}
L(w+s_1u_1+s_2u_2)=L((1+s_1\lambda_1+s_2\lambda_2)w)=(1+s_1\lambda_1+s_2\lambda_2)^2L(y),
\end{equation*}
for small $s_1,s_2\in \mathbb{R}$. Thus 
$(\partial^{2}_{\dot z\dot z } L)_w(u_1,u_2)=2\lambda_1\lambda_2L(w)$.
This fact will be used in the following propositions, where we aim  to   characterize  stationary  and static Finsler spacetimes which are locally standard. 
\end{rmk}
Recalling Remark~\ref{gooddistro}, we assume that $L\in C^1(T\tilde M)$ and we denote by  $\mathcal{D}$ the distribution of hyperplanes in $T\tilde M$ given by $\ker(\partial_{\dot z} L(K))$. 
\bere\label{transv}
If  $L$ is differentiable on $T\tilde M$ and it is fiberwise positively homogeneous of degree $2$,   we have that $(\partial_{\dot z}L)_{K_z}(K_z)=2L(K_z)<0$ hence $(\partial_{\dot z}L)_{K_z}\neq 0$,  $K_z\not\in\mathcal D_z$ and $T_{z}\tilde M=\mathcal{D}_z\oplus [K_z]$, for all $z\in \tilde M$.
\ere
Let us define the map
\[\tilde B(w):= w\in T\tilde M\mapsto \frac 1 2(\partial_{\dot z}L)_w(K_{\pi(w)}).\]
\begin{lem}\label{tildeB}
Let $(\tilde M,L, A)$ be a stationary  Finsler spacetime  with   timelike Killing vector field $K$. 
Assume that $L\in C^1(T\tilde M)$, $L(K)=L(-K)$ and  
\[L(w\pm K_{\pi(w)})=L(w)+L(K_{\pi(w)}),\] 
for all $w\in \mathcal{D}$. Then
\[\tilde B(w)=\frac 1 2 \big (L(w+K_{\pi(w)})-L(w)-L(K_{\pi(w)})\big),\]
for all $w\in T\tilde M$. Moreover, $\tilde B$ is fiberwise linear.
\end{lem}
\begin{proof}
For each $w\in \mathcal D$, let $w_{\mathcal D}\in \mathcal D$ and $\lambda_w\in\R$ be such that $w=w_{\mathcal D}+\lambda_w K_{\pi(w)}$ (recall Remark~\ref{transv}). Moreover, let  $\eps(x)=\mathrm{sign}(x)$, if $x\in \R\setminus\{0\}$, and $\eps(0)=0$.  By definition and our assumptions, we obtain 
\bal
\tilde B(w)&=\frac 1 2\frac{d }{d s}L(w+sK_{\pi(w)})|_{s=0}=\frac 1 2\frac{d }{d s}L(w_{\mathcal D}+\lambda_wK_{\pi(w)}+sK_{\pi(w)})|_{s=0}\nonumber\\
&=\frac 1 2\frac{d }{d s}\left(L(w_{\mathcal D})+(\lambda_w+s)^2L\big(\eps(\lambda_w+s)K_{\pi(w)}\big)\right)\big|_{s=0}\nonumber\\
&=\frac 1 2\frac{d}{d s}\left(L(w_{\mathcal D})+(\lambda_w+s)^2L\big(K_{\pi(w)}\big)\right)\big|_{s=0}\nonumber\\
&=\lambda_wL(K_{\pi(w)}).\label{linearB}
\eal
On the other hand,
\baln
\lefteqn{\frac 12 \big(L(w+K_{\pi(w)})-L(w)-L(K_{\pi(w)})\big)}&\\
&=\frac 12\big( L(w_{\mathcal D}+\lambda_wK_{\pi(w)}+K_{\pi(w)})-L(w_{\mathcal D}+\lambda_wK_{\pi(w)})-L(K_{\pi(w)})\big)\\
&=\frac 1 2 \big(L(w_{\mathcal D})+(\lambda_w+1)^2L(K_{\pi(w)})-L(w_{\mathcal D})-(\lambda_w)^2L(K_{\pi(w)})-L(K_{\pi(w)})\big)\\
&= \lambda_wL(K_{\pi(w)})
\ealn
This proves the first part of the Lemma. 
As $\tilde B(w)=\lambda_w L(K)$ we immediately get that $\tilde B$ is linear in $w$.
\end{proof}

\begin{thm}\label{charstat}
Let $\mathcal T$ be  a line subbundle of $T\tilde M$ and  $(\tilde M,L, T\tilde M\setminus\mathcal T)$ be a stationary  Finsler spacetime  with   timelike Killing vector field $K$. 
Then $(\tilde M,L,T\tilde M\setminus\mathcal T)$ is locally a standard  stationary splitting 
if and only if the following conditions are satisfied:

$(a)$ $L\in C^1(T\tilde M)\cap C^2(T\tilde M\setminus \mathcal T)$ and $K_z\in \mathcal T_z$ at every point $z\in \tilde M$ where $L(z,\cdot)$ is not twice differentiable on $T_z\tilde M$ (so,  at these points $z$, $L(z,\cdot)$ is not the the quadratic form defined by a Lorentzian metric on $T_z\tilde M$.)

$(b)$ $L(K)=L(-K)$;

$(c)$ $L(w\pm K_{\pi(w)})=L(w)+L(K_{\pi(w)})$, for all $w\in \mathcal{D}$.

Furthermore,   it is locally standard static if and only if $(a)$, $(b)$ and $(c)$ hold and $\mathcal D$ is integrable.
\end{thm}
\begin{proof}
($\Rightarrow$) Let $z\in \tilde M$,  $U_z\subset \tilde M$ be a neighborhood of $z$, 
and 
$\phi\colon I_z\times S_z\to \tilde U_z$ be a diffeomorphism such that $\phi_*(\partial_t)=K|_{U_z}$ and $L\circ\phi_*(\tau,v)=-\Lambda \tau^2+\omega(v)\tau+F^2(v)$ (recall Definition~\ref{LSS}). 
If $\bar x\in S_z$ is such that $F^2(\bar x,\cdot)$ is not the square of the norm defined by a Riemannian metric on $T_{\bar x}S_z$,   $L\circ \phi_*\big((t, \bar x), (\cdot,\cdot)\big)$, $t\in I_z$, is not twice differentiable at any vector $(\tau,0)\in\R\times T_{\bar x}S_z$. As $d\phi_{(t,\bar x)}(\partial_t)\equiv d \phi_{(t,\bar x)}(1,0)=K_{\phi(t,\bar x)}$, we deduce that $K_{\phi(t,\bar x)}$ must belong to $\mathcal T_{\phi(t,\bar x)}$ for all $t\in I_z$ and this proves $(a)$. For $(b)$, let $z\in\tilde M$ and take a  map $\phi$ as above (with $z=\phi(0,x)$); then
\begin{equation*}
L(K_z)=L\big(d \phi_{(0,x)}(1,0)\big)=-\Lambda(x)=L\big (d\phi_{(0,x)}(-1,0)\big)=L(-K_z).
\end{equation*}
In order to prove $(c)$, let $y\in \mathcal D_z$ and $(\tau,v)\in \R\times T_{x}M$ be such that $y=d\phi_{(0,x)}(\tau,v)$.   Hence 
\baln
0=(\partial_{\dot z}L)_{K_z}(y) &=(\partial_{\dot z}L)_{d\phi_{(0,x)}(1,0)}(d\phi_{(0,x)}(\tau,v))\\
&=\frac{d}{d s}L\big(d\phi_{(0,x)}(1,0)+sd\phi_{(0,x)}(\tau,v)\big)|_{s=0}\\
&=\frac{d}{d s}L\big(d\phi_{(0,x)}\big(1+s\tau,sv)\big)|_{s=0}\\
&=\frac{d}{d s}\Big(-\Lambda(x)(1+s\tau)^2+2s\omega(v)(1+s\tau)+F^2(sv)\Big)\Big|_{s=0}\\
&=2\big(-\tau\Lambda(x)+w(v)\big);
\ealn
thus $\tau=\omega(v)/\Lambda(x)$ and 
\bmln
L(y\pm K_z)=L\left(d\phi_{(0,x)}\left(\frac{\omega(v)}{\Lambda(x)}\pm 1,v\right)\right)=
\frac{\omega^2(v)}{\Lambda(x)}+F^2(v)-\Lambda(x)\\
=L\left(d\phi_{(0,x)}\left(\frac{\omega(v)}{\Lambda(x)},v\right)\right)+
L\big(d \phi_{(0,x)}(1,0)\big)=L(y)+L(K_z).
\emln
($\Leftarrow$)  Let $\bar z\in \tilde M$ and $S_{\bar z}$ be a small smooth hypersurface in $\tilde M$ such that $\bar z\in S_{\bar z}$ and $T_{\bar z}S_{\bar z}= \mathcal D_{\bar z}$. Recalling Remark~\ref{transv}, we can assume that  $K_x$ is transversal to $S_{\bar z}$, i.e. $T_x \tilde M=T_xS_{\bar z}\oplus [K_x]$, for all $x\in S_{\bar z}$.    From $(b)$ and $(c)$  we get, for any $y, u\in T_{\bar z}S_{\bar z}$,
\bal
\tilde g_y(K_{\bar z},K_{\bar z})=\frac 1 2(\partial_{\dot z\dot z}L)_y(K_{\bar z},K_{\bar z})&=\frac 1 2\frac{\partial^2}{\partial s_1\partial s_2}L(y+(s_1+s_2)K_{\bar z})|_{(s_1,s_2)=(0,0)}\nonumber\\
&=\frac 1 2\frac{\partial^2}{\partial s_1\partial s_2}\big(L(y)+(s_1+s_2)^2L(K_{\bar z})\big)|_{(s_1,s_2)=(0,0)}\nonumber\\
&=L(K_{\bar z})<0,\label{Ktimelike}\eal
and
\bal
\tilde g_y(u,K_{\bar z})=\frac 1 2(\partial_{\dot z\dot z}L)_y(u,K_{\bar z})&=\frac 1 2\frac{\partial^2}{\partial s_1\partial s_2}L(y+s_1u+s_2K_{\bar z})|_{(s_1,s_2)=(0,0)}\nonumber\\
&=\frac{\partial^2}{\partial s_1\partial s_2}\big(L(y+s_1u)+ s_2^2L(K_{\bar z})\big)|_{(s_1,s_2)=(0,0)}=0,\label{Kortho}
\eal
that is $K_{\bar z}$ is timelike w.r.t. the Lorentzian scalar product  $\tilde g_y$ on $T_{\bar z}\tilde M$ and $T_{\bar z}S_z$ is a spacelike hyperplane, for all $y\in T_{\bar z}S_{\bar z}$.
Let  $UTS_{\bar z}$ be the unit tangent bundle of $S_{\bar z}$ (with respect to any auxiliary Riemannian metric on $\tilde M$). As  $\tilde g_y$ is positively homogeneous of degree $0$ in $y$,  by continuity of the map $y\in UTS_{\bar z}\mapsto \tilde g_y$, we get that (up to consider a smaller hypersurface $S_z$) $K_x$ is timelike and $T_x S_{\bar z}$ is spacelike w.r.t. $\tilde g_y$,  for each  $y\in T_xS_{\bar z}$ and for any $x\in S_{\bar z}$.

Let now $w\in T_x\tilde M\setminus \mathcal T_x$, $x\in S_{\bar z}$, and $w_S\in T_{x}S_{\bar z}$, $\tau_w\in\R$ such that $w=w_S+\tau_wK_x$. Let us  evaluate $\tilde g_w(u,u)$, for any  $u\in T_xS_{\bar z}$.  From Lemma~\ref{tildeB} we have
\baln
\tilde g_w(u,u)&=\frac{1}{2}\frac{\partial^2}{\partial s_1\partial s_2}L\big(w+(s_1+s_2)u\big)|_{(s_1,s_2)=(0,0)}\\
&=\frac 1 2 \frac{\partial^2}{\partial s_1\partial s_2}L\big(w_S+\tau_wK_x+(s_1+s_2)u\big)|_{(s_1,s_2)=(0,0)}\\
&=\frac 1 2\frac{\partial^2}{\partial s_1\partial s_2}\Big(L\big(w_S+(s_1+s_2)u\big)\\
&\quad\quad\quad\quad\quad\quad\quad\quad\quad+\tau_w^2L(K_x)+\tilde B\big(w_S+(s_1+s_2)u\big)\Big)\Big|_{(s_1,s_2)=(0,0)}.
\ealn
Since $\tilde B$ is linear,  we have that $\frac{\partial^2}{\partial s_1\partial s_2}\tilde B\big(w_S+(s_1+s_2)u\big)\big|_{(s_1,s_2)=(0,0)}=0$ and then
\beq \tilde g_w(u,u)=\tilde g_{w_S}(u,u)>0.\label{g0}\eeq
By using $w=w_{\mathcal D}+\lambda_w K_x$ and $w_S=(w_S)_{\mathcal D}+\lambda_{w_S}K_x$, we also get, as in \eqref{Ktimelike},
\beq\tilde g_w(K_x,K_x)=L(K_x)=\tilde g_{w_{S}}(K_x,K_x).\label{Lambdax}\eeq
Moreover, recalling \eqref{linearB}, we obtain
\bal
\tilde g_w(u, K_x)&=\frac{1}{2}\frac{\partial^2}{\partial s_1\partial s_2}L\big(w+s_1u+s_2K_x\big)|_{(s_1,s_2)=(0,0)}\nonumber\\
&=\frac 1 2 \frac{\partial^2}{\partial s_1\partial s_2}L\big(w_{\mathcal D}+\lambda_wK_x+s_1u_{\mathcal D}+s_1\lambda_uK_x+s_2 K_x\big)|_{(s_1,s_2)=(0,0)}\nonumber\\
&=\frac 1 2 \frac{\partial^2}{\partial s\partial t}\big(L(w_{\mathcal D}+s_1u_{\mathcal D})+(\lambda _w+s_1\lambda_u+s_2)^2L(K_x)\big)|_{(s_1,s_2)=(0,0)}\nonumber\\
&=\lambda_uL(K_x)=\tilde B(u).\label{omegax}
\eal
Let $I_{\bar z}=(-\varepsilon_{\bar z},\varepsilon_{\bar z})$ be an interval such that the map $\phi:I_{\bar z}\times S_{\bar z}\rightarrow\tilde M$, $\phi(t,x)=\psi_t(x)$, where $\psi$ is the flow of $K$,  is a diffeomorphism onto a neighborhood $U_{\bar z}$ of ${\bar z}$ in $\tilde M$. Consider a non vanishing smooth section $W:S_{\bar z}\rightarrow T\tilde M$, such that $W_x\not\in\mathcal T_x$, for all $x\in S_{\bar z}$. Set $Y_{z}=(d\psi_{t})_{x}(W_x)$, with $z=\phi(t,x)$. So $Y$ is a non vanishing smooth vector field in $U_{\bar z}$.
The evaluation $\tilde g_Y$ of the fundamental tensor of $L$ in $Y$ becomes, then, a Lorentzian metric on $U_{\bar z}$ (and, by definition of $Y$, $K$ is a Killing vector field for $\tilde g_Y$). 
In particular,  $\tilde g_{Y_z}(w_1,w_2)=\tilde g_{W_x}(v_1,v_2)$, for all $z\in U_{\bar z}$, $z=\phi(t,x)$ and 
$w_i=(d \phi)_{(t,x)}(v_i)$, $i=1,2$.
Thus, $\phi^*\tilde g_Y$ in $I_{\bar z}\times S_{\bar z}$ is given by
\[\phi^*\tilde g_Y\big((\tau,v),(\tau,v)\big)=-\tilde g_{W_x}(K_x,K_x)\tau^2+2\tilde g_{W_x}(v,K_x)\tau+\tilde g_{W_x}(v,v).\]
for all $(t,x)\in I_{\bar z}\times U_{\bar z}$ and $(\tau,v)\in \R\times T_xS_{\bar z}$.
From \eqref{g0}, \eqref{Lambdax}, \eqref{omegax} we then obtain:
\baln
\phi^*\tilde g_Y\big((\tau,v),(\tau,v)\big)&=-\tilde g_{(W_x)_S}(K_x,K_x)\tau^2+2\tilde g_{(W_x)_S}(v,K_x)\tau+\tilde g_{(W_x)_S}(v,v)\\
&=-\Lambda(x)\tau^2+\omega_x(v)\tau+g_{(W_x)_S}(v,v),
\ealn 
where $-\Lambda(x)=g_{(W_S)_x}(K_x,K_x)=L(K_x)$ and  $\omega$ is the one-form on $S_{\bar z}$ defined by $\omega=\tilde B|_{TS_{\bar z}}$.
Thus, for all $(\tau, v)\in \R\times TS_{\bar z}\setminus 0$,
\[L\circ\phi_x(\tau,v)=\tilde g_{\phi_*(\tau,v)}\big(\phi_*(\tau,v),\phi_*(\tau,v)\big)=
\Lambda \tau^2+\omega(v)\tau+F^2(v),\]
where $F$ is the Finsler metric on $S_{\bar z}$ defined by
$F(v)=\sqrt{\tilde g_v(v,v)}=\sqrt{L(v)}$ and
\[L\circ\phi_x(\tau,0)=\tau^2L(K_x)=-\Lambda(x)\tau^2,
\]
Hence, for all $(\tau, v)\in \R\times TS_{\bar z}$,
\[L\circ\phi_x(\tau,v)=
\Lambda \tau^2+\omega(v)\tau+F^2(v).\]
This  concludes the proof of the implication to the left. 
For the last part of the theorem, it is enough to observe that in a standard static splitting $(0,v)\in \mathcal D$, for all $v\in TS_{\bar z}$ since $\omega=0$. On the other hand, if $\mathcal D$ is integrable then we can take an integral manifold $S_{\bar z}$ and  then as in \eqref{Kortho} we get $\tilde g_{W_x}(u,K_x)=0$ for all $u\in T_xS_{\bar z}$ and all $x\in S_{\bar z}$. 
\end{proof}
\section{The optical metrics of a stationary splitting Finsler spacetime}
Generally speaking, by {\em optical metric} is meant a metric tensor that comes into play in the description of the optical geometry of a curved spacetime (see, e.g. \cite{HehObu03}). For static and stationary Lorentzian spacetimes, it usually denotes  a Riemannian metric which is conformal to the one induced by the spacetime metric on the space of orbits of the Killing field \cite{AbCaLa88}. For a  standard stationary Lorentzian manifold $(\R\times M,\tilde g)$, $\tilde g$ given in \eqref{sss}, it becomes the  Riemannian metric on $M$ given  by $\omega/\Lambda \otimes \omega/\Lambda +g/\Lambda$. The role attributed to this metric seems to  come from the static case ($\omega=0$) where the metric $g/\Lambda$ does fully describe its optical geometry,   in the sense that  light rays in $(\R\times M,\tilde g)$ project on geodesics of $(M,g/\Lambda)$).  The same is not true in the more general stationary case where the equation satisfied by the projected curves is the one of a unit positive or negative charged  test particle  moving on the Riemannian manifold $(M, \omega/\Lambda \otimes \omega/\Lambda +g/\Lambda)$ under the action of the magnetic field $B=d(\omega/\Lambda)$ (the positive charge correspond to future-pointing lightlike geodesics, the negative one to past-oriented). This equation (actually, these two equations) can effectively be interpreted as the equation of geodesics, parametrized with  constant velocity w.r.t. $\omega/\Lambda \otimes \omega/\Lambda +g/\Lambda$, of a Finsler metric of Randers type on $M$ and of its reverse metric.

Several results about lightlike and timelike geodesics in the standard stationary spacetime can then be deduced by studying geodesics of such  Finsler metrics \cite{CaJaMa11, CaJaMa10, BilJav08, CaJaMa10a, Cap10,  CaJaMa13}. 
The properties of  these Randers metrics encode also  the causal structure  \cite{CaJaMa11, CJS} and the topological lensing  \cite{CapGerSan12, Werner12} in  a standard stationary Lorentzian spacetime, moreover they give information about its c-boundary \cite{FlHeSa13} and its curvature \cite{Gib09}.

Such correspondence between spacetime geometry and Finsler geometry  has also been extended to more general Lorentzian spacetimes introducing  some generalised Finsler-type structures \cite{CJS2}.
Our aim in this section is to prove that the correspondence still holds for a wide class of stationary splitting Finsler spacetimes $(\R\times M,L)$.

Observe that if $\gamma=(\theta,\sigma)$ is a lightlike curve on $(\R\times  M,L)$, $L$ defined as in \eqref{stationary} then, by definition, it satisfies the equation
\begin{equation*}
0=L(\dot{\gamma})=-\Lambda(\sigma)\dot\theta^2+2B(\dot\sigma)\dot\theta+F^2(\dot\sigma),
\end{equation*}
so
\[
\dot\theta=\frac{B(\dot\sigma)}{\Lambda(\sigma)}\pm\sqrt{\frac{B(\dot\sigma)^2}{\Lambda^2(\sigma)}+\frac{F^2(\dot\sigma)}{\Lambda(\sigma)}}.
\]
Solving the above equation w.r.t. to $\dot\theta$, we get the following non-negative and fiberwise positively homogeneous of degree 1 Lagrangians on $TM$ associated to the stationary splitting Finsler spacetime $(\R\times M,L)$:
\begin{equation}\label{FBFB-}
\begin{split}F_B&=\frac{B}{\Lambda}+\sqrt{\frac{B^2}{\Lambda^2}+\frac{F^2}{\Lambda}},\\ F^{-}_{B}&=-\frac{B}{\Lambda}+\sqrt{\frac{B^2}{\Lambda^2}+\frac{F^2}{\Lambda}}.
\end{split}
\end{equation}
The same assumptions ensuring that $L$ is a Lorentz-Finsler function on $\tilde M=\R\times M$, plus a definite sign of $B$ when it does not reduce to a one-form on $M$, give that $F_B$ and $F^-_B$ are  Finsler metrics on $M$.
\begin{thm}\label{fermatmetrics}
Let $L\colon T\tilde M\to \R$ defined as in \eqref{stationary}. Assume that  \begin{enumerate}
	\item $(\partial^2_{yy}B)_v$ is positive semi-definite (resp.  $(\partial^2_{yy}B)_v$ is negative semi-definite), for all $v\in TM\setminus 0$;
	\item for each $x\in M$ either $B(x,v)\geq 0$ (resp. $B(x,v)\leq 0$) for all $v\in T_xM$  or $B(x,\cdot )$ is linear on $T_xM$;
\end{enumerate} 
then  $F_B$ (resp. $F^{-}_B$) is a Finsler metric on $M$.
\end{thm}
\begin{proof}
Let us prove the statement for $F_B$, i.e. in the case when $(\partial^2_{yy}B)_v$ is positive semi-definite and either $B(x,\cdot )$ is non-negative or it  is linear. 
The only non trivial part of the proof is to show that  the fiberwise Hessian of the square of  $F_B$ is positive definite.
Let us define
\beq G=\sqrt{B^2 +\Lambda F^2},\label{Gfinsler}\eeq
so that $F_B=\frac{1}{\Lambda}(B+G)$. Let us equivalently compute the fiberwise Hessian of $\frac 1 2 (\Lambda F_B)^2$ at $v\in TM\setminus 0$:
\bml\frac{\Lambda^2}{2}(\partial^2_{yy}F^2_B)_v=\big((\partial_yB)_v+(\partial_yG)_v)\big)\otimes\big((\partial_yB)_v+(\partial_yG)_v)\big)
\\+(B(v)+G(v))\big((\partial^2_{yy}B)_v+(\partial^2_{yy}G)_v\big). \label{lambdaF}
\eml
Let us now show that $G$ is a Finsler metric on. Clearly, $G$ is non-negative and vanishes only at zero vectors, it is continuous in $TM$ and smooth outside the zero section, moreover it is fiberwise positively homogeneous of degree $1$. It remains to prove that the fiberwise Hessian of $\frac{1}{2}G^2$ is positive definite on $TM$.  
Let us evaluate, 
\[\frac 1 2 (\partial^2_{yy}G^2)_v=(\partial_yB)_v\otimes (\partial_yB)_v+B(v)(\partial^2_{yy}B)_v+\frac{\Lambda}{2}(\partial^2_{yy} F^2)_v.
\]
As $(\partial_yB)_v\otimes (\partial_yB)_v+\frac{\Lambda}{2}(\partial^2_{yy} F^2)_v$ is positive definite, we see that $\frac 1 2 (\partial^2_{yy}G^2)_v$ is positive definite provided that $B$ is linear on $T_{\pi_M(v)}M$ (thus, $(\partial^2_{yy}B)_v=0$ for all $v\in T_{\pi_M(v)}M\setminus 0$) or $B(v)\geq 0$ and $(\partial^2_{yy}B)_v$ is positive semi-definite. 

Since $G$ is a Finsler metric, we know that $(\partial^2_{yy}G)_v$ is positive semi-definite and $(\partial^2_{yy}G)_v(u,u)=0$ if and only if $u=v$.  
Thus, as $B(v)+G(v)\leq 0$ and it vanishes only if $v=0$, from \eqref{lambdaF}, we see   $\frac{\Lambda^2}{2}(\partial^2_{yy}F^2_B)_v$ is positive semi-definite. Let us then assume, by contradiction that there exist $u\in T_{\pi_M(v)} M$, $u\neq 0$,  such that $\frac{\Lambda^2}{2}(\partial^2_{yy}F^2_B)_v(u,u)=0$. This implies that
$(\partial^2_{yy}G)_v(u,u)=0$ and then $u=v$. Hence, by homogeneity, $(\partial^2_{yy}B)_v(v,v)=0$ and \[0=\big((\partial_yB)_v+(\partial_yG)_v)\big)\otimes\big((\partial_yB)_v+(\partial_yG)_v)\big)(v,v)=(B(v)+G(v))^2,\]
which implies that $v=0$, a contradiction.

If $(\partial^2_{yy}B)_v$ is negative semi-definite and $B(x,\cdot )$ is non-positive or linear, we get  that $F_B^-$ is   a Finsler metric on $M$ as above taking into account that 
$F_B^-=\frac{1}{\Lambda}(G-B)$.
\end{proof}
\bere
A part from the case when $B$ is a one-form on $M$, a significant class of maps satisfying the assumptions of Theorem~\ref{fermatmetrics} is given (up to the sign) by Randers variations of  Finsler metrics: $B=\pm (\omega+F_1)$, where $F_1$ is a Finsler metric on $M$ and $\omega$ is a one-form such that $|\omega(v/F_1(v))|<1$, for all $v\in TM\setminus 0$.
\ere

\section{Trivial isocausal static Finsler spacetimes}
The optical metrics have been already introduced in \cite{CapSta16} for  a standard static Finsler spacetime. Since in this case $B=0$, both $F_B$ and $F_B^-$ reduce to the Finsler metric $F/\Lambda$ on $M$.  The causal properties of a standard static Finsler spacetime   can then be described in terms of metric properties of  $F/\Lambda$. Thanks to the metrics $F_B$ and $F_B^-$,  we easily see that the causal properties of  a  stationary splitting Finsler spacetime coincides with the ones of a couple of  standard static Finsler spacetimes and therefore they can be described using $F_B$ and $F^-_B$. Indeed, under the assumptions of Theorem~\ref{fermatmetrics},  we can  consider the   Lorentz-Finsler functions on $\tilde M$ given by
\begin{equation}\label{LBLB-}
L_B(\tau,v):=-\tau^2+F_{B}^2(v), \quad\quad L_{B^-}(\tau,v):=-\tau^2+(F^-_{B})^2(v),
\end{equation}
and the Finsler spacetimes $(\R\times M, L_B,T^+\tilde M\setminus \mathcal T )$,  $(\R\times M, L_{B^-},T^-\tilde M\setminus \mathcal T)$ that we call  \emph{trivial isocausal static Finsler spacetimes} associated to $(\tilde M, L)$. 
Isocausality is a relation between Lorentzian spacetimes  introduced in \cite{GarSen03}. If $V$ and $W$ are spacetimes they are said {\em isocausal} if there exists  two diffeomorphisms $\varphi:V\to W$ and $\psi:W\to V$ such that $\varphi_*$ and $\psi_*$ map future-pointing causal vectors into future pointing- causal vectors. Clearly this notion make sense also for Finsler spacetimes. However, notice that in the Finsler setting  future and past-pointing causal vectors are in general not related by the symmetry $v\mapsto -v$ and so one should consider separately   causal relations involving future and past, provided that both the future and past causal cones are defined.  In our case, we have that $(\tilde M, L, T^+\tilde M\setminus \mathcal T)$ and $(\tilde M, L, T^-\tilde M\setminus \mathcal T)$ are both trivially related (i.e. $\varphi=\psi=i_{\R\times M}$) respectively to $(\R\times M, L_B,T^+\tilde M\setminus \mathcal T )$ and $(\R\times M, L_{B^-},T^-\tilde M\setminus \mathcal T)$ as the following proposition shows. 
\begin{prop}\label{samecausality}
Let $(\tilde M, L)$ be a stationary splitting Finsler spacetime satisfying the assumptions of Theorem~\ref{fermatmetrics}.
Then $(\tau,v)\in T\tilde M$, with $\tau>0$ (resp. $\tau<0$)  is a  causal  vector of $(\tilde M,L)$  if and only if it is a causal, future-pointing (resp. past-pointing)  vector of   $(\tilde M,L_B)$ (resp. $(\tilde M,L_{B^-})$). 
\end{prop}
\begin{proof}
By definition, a non-zero vector $(\tau,v)\in T\tilde M$ is causal for $(\tilde M, L)$ if and only if $L(\tau,v)\leq 0$ i.e.
$-\Lambda\tau^2+2B(v)\tau+F^2(v)\leq 0$
which is equivalent to
$\tau\ge F_B(v)$ or $\tau\leq-F^{-}_B(v)$. Then the equivalence follows recalling that (see Remark~\ref{futurepastpoint}), a causal, future-pointing (resp. past-pointing) vector in  an ultra-static standard Finsler spacetime $(\tilde M, L_1)$, $L_1(\tau,v)= -\tau^2 +F_1^2(v)$, where  $Y=\partial_t$,  is a non-zero vector $(\tau,v)$ such that $\tau\geq F_1(v)$ (resp. $\tau\leq -F_1(v)$). 
\end{proof}
\bere \label{futureinside}
In particular, the statement of Proposition~\ref{samecausality} holds with the words ``timelike'' or  ``lightlike'' replacing ``causal''.
Notice also that the causal future-pointing vectors of $(\tilde M, L, T^+\tilde M\setminus \mathcal T)$ (resp. of $(\tilde M, L, T^-\tilde M\setminus \mathcal T))$ are all and only the causal vectors in $\bar T^+\tilde M:=\cup_{p\in \tilde M}\bar T^+_p\tilde M$   (resp. $\bar T^-\tilde M:=\cup_{p\in \tilde M}\bar T^-_p\tilde M$). In fact,  $(\tau, v)\in T^+\tilde M$ (resp. $(\tau, v)\in T^-\tilde M$) is causal for $(\tilde M, L, T^+\tilde M\setminus \mathcal T)$ (resp. $(\tilde M, L, T^-\tilde M\setminus \mathcal T)$) if and only
$\tau \geq F_B (v)$ (resp. $\tau\leq -F_{B}^-$). As  $F_B(v)=\frac{1}{\Lambda}(B(v)+G(v))>\frac{B(v)}{\Lambda}$ (resp. $-F_B^-(v)=\frac{1}{\Lambda}(B(v)-G(v))<\frac{B(v)}{\Lambda}$), provided that $v\neq 0$, we conclude by recalling Remark~\ref{futurepastpoint}.

When $B$ is a one-form on $M$ both the sets of future and past-pointing vectors of $(\tilde M, L, T\tilde M\setminus \mathcal T)$ are non empty and, being in this case $Y=\partial_t$, they coincide with with the causal vectors belonging respectively in  $\bar T^+\tilde M$ and $\bar T^-\tilde M$.
\ere
\begin{rmk}\
As in Lemma 2.21 in \cite{CapSta16}, the epigraph (resp. hypograph) of the function $\tau=F_B (v)$  (resp. $\tau=-F^-_B (v)$) in $T_p \tilde M$ is connected and convex, for all $p\in \tilde M$, i.e. the set of the future-pointing  causal vectors  of $(\tilde M, L, T^-\tilde M\setminus \mathcal T)$ (resp. $(\tilde M, L, T^-\tilde M\setminus \mathcal T)$) at $p\in \tilde M$ is connected and convex. 
Moreover, for each $c>0$, the set  of the future-pointing timelike vectors $(\tau,v)$ of $(\tilde M, L, T^-\tilde M\setminus \mathcal T)$ (resp. $(\tilde M, L, T^-\tilde M\setminus \mathcal T)$) in $T_{p}\tilde M$ such that  $L(\tau, v)\leq -c\}$ is also connected and strictly convex.
\end{rmk}
Let $p_0\in\tilde M$ and let us denote by $I^+(p_0)$ (resp. $I^-(p_0)$) the subset of $\tilde M$ given by all the points $p\in \tilde M$ such that there exists a timelike future-pointing curve of $(\tilde M,L)$ connecting $p_0$ to $p$ (resp. $p$ to $p_0$). 
From Proposition~\ref{samecausality} and Remark~\ref{futureinside}, it follows that these sets coincide with those  of the corresponding trivial isocausal static Finsler spacetime. Thus from  \cite[Prop. 3.2]{CapSta16} we obtain:
\begin{prop}\label{causality1}
Let $(\tilde M, L, T^+\tilde M\setminus \mathcal T)$  be a stationary splitting Finsler spacetime (thus $(\partial^2_{yy}B)_v$ is positive semi-definite,  for all $v\in TM\setminus 0$) such that, for each $x\in M$, either $B(x,v)\geq 0$, for all $v\in T_xM$,  or $B(x,\cdot )$ is linear on $T_xM$. Then, for all $p_0=(t_0,x_0)\in \tilde M$ we have:
\[I^{+}(p_0)=\bigcup_{r>0}\left(\{t_0+r\}\times B^{+}(x_0,r)\right), \quad I^{-}(p_0)=\bigcup_{r>0}\left(\{t_0-r\}\times B^{-}(x_0,r)\right),\]
where $B^{+}(x_0,r)$ and $B^-(x_0,r)$ denote, respectively, the forward and the backward open ball of centre $x_0$ and radius $r$ (see, e.g. \cite{Chern} for the definitions of these balls) of the Finsler metric $F_B$.  Moreover, $I^{\pm}(p_0)$ are open subsets of $\tilde M$.
\end{prop}
\bere Taking into account that, in $(\tilde M, L, T^-\tilde M\setminus \mathcal T)$, under the assumptions $(\partial^2_{yy}B)_v$ is negative  semi-definite,  for all $v\in TM\setminus 0$ and, for each $x\in M$, either $B(x,v)\leq 0$, for all $v\in T_xM$,  or $B(x,\cdot )$ is linear on $T_xM$,  a timelike future-pointing vector  is past-pointing for the standard static Finsler spacetime $(\tilde M, L_{B^-})$, an analogous proposition holds  for  a stationary splitting spacetime of the type $(\tilde M, L, T^-\tilde M\setminus \mathcal T)$ by considering the forward and 
backward ball of the metric $F^-_{B}$ and replacing $t_0+r$ with $t_0-r$ in the first equality and $t_0-r$ with $t_0+r$ in the second one.
\ere
Analogously, using also the Fermat's principle  (see Appendix~\ref{fermatprinc}), from the results in Section 3 of \cite{CapSta16}, we get the following proposition which extends to stationary splitting Finsler spacetimes some results obtained in \cite{CJS} for  standard stationary Lorentzian spacetimes (we refer to \cite{CapSta16} and to  \cite{Minguzzi Sanchez} for the definitions of the causality properties involved in its statement, while for the notions of forward and backward completeness of a Finsler metric we refer to \cite{Chern}).
\begin{prop}\label{causality2} 
Under the assumptions of Theorem~\ref{fermatmetrics}, let $(\tilde M, L,T^+\tilde M\setminus \mathcal T)$ (resp. $(\tilde M, L,T^-\tilde M\setminus \mathcal T)$) be a stationary splitting Finsler spacetime.  Then the following propositions hold true:
\begin{enumerate}
	\item $(\tilde M, L,T^+\tilde M\setminus \mathcal T)$ (resp. $(\tilde M, L,T^-\tilde M\setminus \mathcal T)$) is causally simple if and only if for any $x,y\in M\times M$ there exists a geodesic of $F_B$ (resp. $F^-_B$) joining $x$ to $y$, with length equal to the distance associated to $F_B$ (resp.  $F_B^-$);
	\item a slice (and then any slice) $S_t=\{t\}\times M$ is a Cauchy hypersurface  if and only the Finsler manifold $(M,F_B)$ (resp. $(M,F^-_B)$) is forward and backward complete;
	\item $(\tilde M, L,T^+\tilde M\setminus \mathcal T)$ (resp. $(\tilde M, L,T^-\tilde M\setminus \mathcal T)$) is globally hyperbolic if and only if $ \bar B^{+}(x,r)\cap \bar B^{-}(y,s)$ is compact, for every $x,y\in M$ and $r,s>0$, where $\bar B^{\pm}(x_0,r_0)$ are the  closure of the forward and backward balls on $M$ associated to  metric $F_B$ (resp. $F_B^-$).
	\end{enumerate}
\end{prop}
\section{Conclusions}
In this work, we have introduced a Lorentz-Finsler function, Eq. \eqref{stationary}, on $\R\times M$   which admits a timelike Killing vector field and  can be considered as a natural generalisation of the quadratic form of a  standard stationary Lorentzian metric. We have characterized when a Finsler spacetime with a timelike Killing vector field is locally of the type introduced here. Moreover, we have seen that the  optical geometry of these Finsler spacetimes can be described by (at least) one  of two classical Finsler metrics on $M$, Eq. \eqref{FBFB-},  leading to causal relations between  this class of stationary Finsler spacetimes and $\R\times M$ endowed with two possible static Lorentz-Finsler functions, Eq. \eqref{LBLB-},  corresponding respectively to some sign assumptions on $B$ and its fiberwise Hessian.  These relations hold  also for a standard stationary spacetime, Eq. \eqref{sss}, as its optical  metrics are Finslerian too: $F_B$ is a Randers metric and $F_{B}^{-}$ is its reverse metric (see \cite{CaJaMa11}), showing that isocausality can hold between Lorentzian and Finsler spacetimes as well. 

The  isocausality between  stationary splitting Finsler spacetimes and  static Finsler spacetimes allowed us to deduce some results about the former class from already known ones valid for  the latter \cite{CapSta16}, as shown in Propositions~\ref{causality1} and \ref{causality2}. In particular these results hold whenever the map $B$ reduces to a one form on $M$. Thus, imitating  the modification of the Schwarzschild metric in \cite{LPH},  a  Finslerian perturbation of the Kerr metric  given (in geometric units) as
\[
L(\tau,\dot r,\dot\theta,\dot\varphi):=
-\Big(1 - \frac{2Mr}{\rho^2}+\psi_0(r)\Big)\tau^2-\frac{2Mra\sin^2\theta}{\rho^2}\tau\dot\varphi + F^2(\dot r,\dot\theta,\dot\varphi),\]
where 
\bmln
F(\dot r,\dot\theta,\dot\varphi):=
\Big(\big(\frac{\rho^4}{\Delta^2}+\psi_1(r)\big)\dot r^4
+\big(\rho^4+\psi_2(r)\big)\dot\theta^4 \\+\big((r^2+a^2+\frac{2Ma^2\sin^2\theta}{\rho^2}\big)^2\sin^4\theta+\psi_3(r)\big)\dot\varphi^4  
\Big)^{1/4}, 
\emln
$\rho^2:=r^2+a^2\cos^2\theta$,  $\Delta:=r^2-2Mr+a^2$, $(r,\theta, \varphi)$ are spherical coordinates  on $\R^3$ and $(\tau, \dot r, \dot\theta, \dot\varphi)$  the induced ones  on  $T_{p}\R^4$, for each $p=(t,r,\theta,\varphi)\in\R^4$, $M>0$ and $a\neq 0$,
belongs (on the region where $1 - \frac{2Mr}{\rho^2}+\psi_0(r)>0$ and for small enough functions $\psi_i(r)$) to the class of Lorentz-Finsler functions for which our results hold. 
Nevertheless, more general Lorentz-Finsler functions $L$ can be considered by changing the Finsler metric $F$ on $\R^3$. In particular, as $F$ can be taken non-reversible ($F(\dot r,\dot\theta,\dot\varphi)\neq F(-\dot r,-\dot\theta,-\dot\varphi)$), the frame dragging effect,
which emphasizes the bending angle of light rays that propagate in the direction of rotation of the Kerr black hole \cite{amir2, IyeHan09, IyeHan09b}, might be  fine-tuned by a suitable choice of a Finsler metric $F$ depending in a non symmetric way from $\dot\varphi$.

An example of a stationary splitting  Finsler spacetime appeared as a solution of the Finslerian gravitational field equations proposed by S. F. Rutz in \cite{Rutz} as a generalisation of the Einstein field equations in vacuum. It is described by the spherical symmetric Lorentz-Finsler function (compare also with \cite[Eq. (37)]{Bar12}) 
\bmln
L(\tau,\dot r,\dot\theta,\dot\varphi):=-\Big(1 - \frac{2M}{r}\Big)\tau^2+\Big(1 - \frac{2M}{r}\Big)^{-1}\dot r^2+r^2(\dot\theta^2+\sin^2\theta\dot\varphi^2)\\+\epsilon\Big(1 - \frac{2M}{r}\Big)\tau B(\dot r,\dot\theta,\dot\varphi),
\emln
where $B(\dot r,\dot\theta,\dot\varphi):=(\dot\theta^2+\sin^2\theta\dot\varphi^2)^{1/2}$. Observe that when the  parameter $\epsilon$ is equal to $0$, $L$ reduces to the quadratic form associated to the Schwarzschild metric. The map $B$ satisfies the assumption of Theorem~\ref{fermatmetrics} only on the open region defined by $0<\theta<\pi$ and, for each $(r,\theta,\varphi)\in \big((0,+\infty)\setminus\{2M\}\big)\times (0,\pi)\times [0,2\pi]$, on the cone $\dot\theta^2+\sin^2\theta\dot\varphi^2>0$ in $\R^3$, where $B$ admits fiberwise Hessian. This example suggests  that it would be interesting to weaken our setting allowing that $B$ is twice differentiable only on  a cone subset $A_M$ of $TM\setminus 0$, giving rise to  two conic (in the sense of \cite{JavSan14}) optical metrics $F_B$ and $F^-_B$, smooth  only on $A_M$. This can be further generalized starting from a conic  Finsler metric $F$ or to a  Killing vector field $\partial_t$ which is not timelike everywhere \cite{CJS2}. However, the study of causality, geodesics and  Fermat's principle  in the corresponding  Lorentz-Finsler  spacetime would be much more delicate than the case where $F$ is a standard Finsler metric on $M$.
\section*{Acknowledgements}
We are grateful to the referees for their valuable comments. 
We would also like to thank Miguel Angel Javaloyes for some fruitful conversations about Killing fields for Finsler metrics and Miguel S\'anchez for suggesting to study the case when the map $B$ is more general  than a one-form.
\appendix
\section{Geodesics}
The fact that the Lorentz-Finsler function of a Finsler spacetime is  regular  (in the sense of the calculus of variations) only on the cone subset  $A\subset T\tilde M$ poses some problems for the definition of geodesics. 
A shortcut leading to the geodesics equation as the Euler-Lagrange equation of the energy functional $E=\gamma\mapsto \frac 12 \int_a^b L(\dot{\gamma})d s$ consists in considering  only smooth curves $\gamma\colon[a,b]\to\tilde M$, with velocity  vector  field $\dot{\gamma}(s)\in A$ for all $s\in[a,b]$) and smooth variations (recall that  $A$ is open, thus for small values of the  parameter  of the variation the ``longitudinal'' curves have velocity vectors in $A$).  Then we can  give the following:
\begin{dfn}\label{geodef}
Let $(\tilde M, L, A)$ be a Finsler spacetime. A  smooth curve $\gamma\colon[a,b]\to\tilde M$ such $\dot\gamma(s)\in A$ for all $s\in [a,b]$ is  a (affinely parametrized) {\em geodesic}  if in local natural coordinates $(x^0,\ldots, x^n,y^0,\ldots, y^n)$ of $T\tilde M$ it satisfies the equations 
\beq\label{E-L}\frac{\partial L}{\partial x^i }\big(\gamma(s), \dot\gamma(s)\big)=\frac{d}{ds}\left(\frac{\partial L}{\partial y^i}\big (\gamma(s), \dot\gamma(s)\big)\right),\quad i=0,\ldots,n.\eeq
\end{dfn}
Since $\frac{\partial^2 L}{\partial y^i\partial y^j}(x,y)$ is non-degenerate for all $x\in \tilde M$ and $y\in A_x$, \eqref{E-L} can be put in normal form (as a second order ODE) and then we have also that for each initial condition $(x_0,y_0)\in A$ there exists one and only geodesic defined in a neighborhood of $0\in\R$ and such that $\gamma(0)=x_0$, $\dot\gamma (0)=y_0$.

Moreover, from $s$-independence of the Lagrangian $L$ and the fact that it is positively homogeneous of degree $2$, we know that  there exists a constant $C_\gamma\in\R$ such that  $L(\dot\gamma(s))=(\partial_y L)_{\dot\gamma(s)}(\dot\gamma(s))-L(\dot\gamma(s))=C_\gamma$. Hence geodesics in Finsler spacetime have a well defined causal character: they are timelike, lightlike or spacelike according to $C_\gamma<0$, $C_\gamma=0$, $C_\gamma>0$.

However, in many geometrical and analytical problems, considering only smooth curves is not optimal being $H^1$  or ``piecewise smooth''  curves more  convenient.  In general, a key ingredient in order to prove that a $H^1$ or a piecewise smooth critical point of the energy is  smooth  is the injectivity of the Legendre map $v\in TM \mapsto (\partial_{\dot z}L)_v\in T^*M$.  Injectivity  could  be proved  for timelike or lightlike vectors ($L(v)\leq 0$, $v\in A$) adapting  \cite[Theorem 5]{Minguzzi}, being then  enough to prove regularity of lightlike of timelike geodesics but not for spacelike ones ($L(v)>0$, $v\in A$).   We observe that in a stationary splitting Finsler spacetime, if $B$ is not a one-form,  the Legendre map is not defined on $\mathcal T$ and a proof of its  global injectivity  following  \cite[Theorem 6]{Minguzzi} doesn't seem to extend immediately  to our setting due both to the possible degeneration of $\tilde g$ outside $A$ and the lack of compactness of $S_x:= (T_x\tilde M\setminus \mathcal T_x)\cap \mathbb S^{n}$, where $\mathbb S^n$ is the sphere in $T_x \tilde M$, $x\in \tilde M$. 
Nevertheless, we would like to write down the geodesics equation taking into account  the splitting  $\tilde M=\R\times M$. Let us then introduce the following setting.

Let $\tilde M=\R\times M$ and $(\tilde M, L)$ be a stationary splitting spacetime. Let $x_0, x_1\in  M$ and $\Omega^r_{x_0x_1}(M)$ be the set of the continuous,  piecewise smooth, regular curves  $\sigma$ on $M$, parametrized on a given  interval $[a,b]\subset \R$ and connecting $x_0$ to $x_1$ (i.e. $\sigma(a)=x_0$, $\sigma(b)=x_1$). By {\em regular} here we mean that the left and right derivatives $\dot\sigma^{\pm}(s)$ are different from $0$ for all $s\in [a,b]$. 
%
Let $t_0, t_1\in\R$ and let $\Omega_{t_0t_1}(\R)$ be the  set of the continuous,  piecewise smooth functions $t$ defined on $[a,b]$ such that  $t(a)=t_0$ and $t(b)=t_1$. For $p_0,p_1\in \tilde M$, $p_0=(x_0,t_0)$, $p_1=(x_1,t_1)$, let us set 
\[\Omega^r_{p_0p_1}(\tilde M):=\Omega_{t_0t_1}(\R)\times\Omega^r_{x_0x_1}(M).\]
If $\gamma\in\Omega^r_{p_0p_1}(\tilde M)$, we call a {\em (proper) variation} of $\gamma$ a continuous two-parameter map $\psi\colon (-\varepsilon,\varepsilon)\times [a,b]\to \tilde M$ such that $\psi(0,s)=\gamma(s)$, for all $s\in [a,b]$, $\psi(r,\cdot)$ is a continuous curve between  $p_0$ and $p_1$, for all $r\in (-\varepsilon,\varepsilon)$, and  there exists a subdivision $a=s_0<s_1<\ldots, s_k=b$ of the interval $[a,b]$ for which  $\psi|_{(-\varepsilon,\varepsilon)\times[s_{j-1},s_j]}$ is smooth for all $j\in \{1,\ldots,k\}$.
Clearly, we can define classes of proper variations of $\gamma$ as those sharing the same  {\em variational vector field} $Z$. This is, by definition,  a continuous piecewise smooth vector field along $\gamma$ such that $Z(a)=0=Z(b)$ and $Z(s)=\frac{\partial \psi}{\partial r}(0,s)$. By considering any auxiliary Riemannian metric $h$ on $\tilde M$, we see that   each variational vector field $Z$ along $\gamma$ individuates a variation (and then also a class of them) by setting $\psi(w,s):=\exp_{\gamma(s)}(wZ(s))$, for $|w|<\varepsilon$ small enough.
 Moreover, as $\sigma$ is regular and $\psi|_{(-\varepsilon,\varepsilon)\times[s_{j-1},s_j]}$ is  smooth, we can assume, up to consider a smaller $\varepsilon$, that all the curves $\psi(r,\cdot)$ have component on $M$ which is a regular curve. 
Of course, if $\gamma$ is a geodesic of the stationary splitting spacetime $(\tilde M, L)$ then it is a critical point of $E$ in $\Omega^r_{p_0p_1}(\tilde M)$, i.e. $\frac{d}{dr}(E(\psi(r,\cdot))|_{r=0}=0$, for all proper variations $\psi$ of $\gamma$.  Let us study the converse implication. 
\begin{prop}\label{geosforsplitting}
	Let $(\tilde M, L)$ be a stationary splitting Finsler spacetime and $\gamma\in\Omega^r_{p_0p_1}(\tilde M)$, $\gamma(s)=(\theta(s),\sigma(s))$, be a critical point of $E$ in $\Omega^r_{p_0p_1}(\tilde M)$ then:
	\begin{enumerate}
		\item there exists a constant $c_\gamma$ such that 
		\begin{equation}\label{Killingconst}
		-\Lambda(\sigma)\dot{\theta}+B(\dot\sigma)=c_\gamma; 
		\end{equation}
		\item assume that $n=\mathrm{dim}(M)\geq 3$,  $B\geq 0$, $(\partial^2_{yy} B)_v$ is positive semi-definite for all $v\in TM\setminus 0$ and  $c_\gamma\leq 0$ (resp. $B\leq 0$, $(\partial^2_{yy} B)_v$ is negative semi-definite for all $v\in TM\setminus 0$ and $c_\gamma\geq 0$) then $\sigma$ is smooth (then from \eqref{Killingconst} also $\theta$ is smooth) and,
			 in local natural coordinates $(x^1,\ldots, x_n,  y^1,\ldots,y^n)$ on $TM$, the following equations are satisfied:
\begin{equation}\label{onM}
-\frac 1 2 \Big(\frac{c_\gamma}{\Lambda(\sigma)}\Big)^2\frac{\partial \Lambda}{\partial x^i}(\sigma)+\frac{\partial H_\gamma}{\partial x^i}(\dot\sigma)-\frac{d}{ds}\left(\frac{\partial H_\gamma}{\partial y^i}(\dot\sigma)\right)=0,\quad i=1,\ldots,n\end{equation}
where  $H_\gamma=-c_\gamma \frac{B}{\Lambda} +\frac 1 2 \left(\frac{B^2}{\Lambda}+F^2\right)$.
	\end{enumerate}
\end{prop}
\bere\label{causalgeo}
In a stationary splitting Finsler spacetime Eqs. \eqref{Killingconst} and \eqref{onM} are  equivalent to \eqref{E-L} if we consider  smooth  curves $\gamma=(\theta,\sigma)$, with regular components  $\sigma$, and smooth variation. The assumptions $B\geq 0$, $(\partial^2_{yy} B)_v$ positive semi-definite (resp. $B\leq 0$, $(\partial^2_{yy} B)_v$ negative semi-definite) for all $v\in TM\setminus 0$  are compatible with the case where $B$ is not a one-form and $A=T^+\tilde M\setminus \mathcal T$ (resp. $A=T^+\tilde M\setminus \mathcal T$) -- recall Proposition~\ref{index1}.   The additional assumption $c_\gamma\leq 0$ (resp. $c_\gamma\geq 0$) ensures that $\dot{\theta}\geq 0$ (resp. $\dot{\theta}\leq 0$) and therefore, apart from the case when $c_\gamma=0$ (where $\dot\theta(s)$ could vanish at some instants)   $\gamma$ in Proposition~\ref{geosforsplitting} is a geodesic of $(\tilde M, L, A)$.
Recall that from Remark~\ref{futureinside}, any causal curve $\gamma=(\theta, \sigma)$ of $(M,L,T^+\tilde M\setminus\mathcal T)$ (resp. $(M,L,T^-\tilde M\setminus\mathcal T)$) with $\dot\gamma(s)\in T^+\tilde M\setminus\mathcal T$ (resp. $\dot\gamma(s)\in T^-\tilde M\setminus\mathcal T$) satisfies  $-\Lambda(\sigma)\dot{\theta}+B(\dot\sigma)< 0$ (resp. 	$-\Lambda(\sigma)\dot{\theta}+B(\dot\sigma) > 0$) and therefore we have that any causal curve $\gamma\in \Omega^r_{p_0p_1}(\tilde M)$ which is a critical point of $E$, is a geodesic.   
 \ere
Before proving Proposition~\ref{geosforsplitting} we need the following lemma:
\begin{lem}\label{legendre}
	Let  $\alpha\in\R$ and $H_\alpha=-\alpha \frac B \Lambda+\frac 12 \left (\frac{B^2}{\Lambda}+F^2\right)$. If 	
$B\geq 0$, $(\partial^2_{yy} B)_v$ is positive semi-definite, for all $v\in TM\setminus 0$, and  $\alpha \leq 0$ (resp.$B\leq 0$, $(\partial^2_{yy} B)_v$ is negative semi-definite for all $v\in TM\setminus 0$ and  $\alpha \geq 0$), then $(\partial^2_{yy}H)_v$ is positive definite for all $v\in TM\setminus 0$. Moreover, if $n=\mathrm{dim}(M)\geq 3$, 
the map 
\[\mathcal L = v\in TM\setminus 0\mapsto (\partial_y H_\alpha)_v\in T^*M\setminus 0\]
is a diffeomorphism.
\end{lem}
\begin{proof}
Let us prove the statement under the assumptions $B\geq 0$, $(\partial^2_{yy} B)_v$  positive semi-definite,  $\alpha \leq 0$, being the other case analogous.
Observe that if there exists $v\in TM\setminus 0$ such that $(\partial_y H_\alpha)_v=0$ then by homogeneity $0=(\partial_y H_\alpha)_v(v)=-\frac{\alpha}{\Lambda} B(v)+ \frac{1}{\Lambda}B^2(v)+F^2(v)$ hence, being $- \frac{\alpha}{\Lambda}B(v)\geq 0$, it must be $F(v)=0$ and then $v=0$. Thus $\mathcal L$ is a continuous map from $TM\setminus 0$ into $T^*M\setminus 0$. 
Let us observe that $B^2/\Lambda+F^2=G^2/\Lambda$ where $G$ is the map in \eqref{Gfinsler}.
As $(\partial^2_{yy} H)_v= -\frac{\alpha}{\Lambda} (\partial^2_{yy} B)_v + \frac{1}{2\Lambda}(\partial^2_{yy} G^2)_v$ and $G$ is a Finsler metric (recall the proof of Theorem~\ref{fermatmetrics}) we get that $\mathcal L$ is a local diffeomorphism.

Moreover,  since for any $x\in M$  the map $v \in T_xM\setminus \{0\}\mapsto (\partial_{y} B)_v\in T^*_x M$ is positively homogeneous  of degree $0$ and  $v \in T_xM\setminus \{0\}\mapsto (\partial_{y} G^2)_v\in T^*_xM$ is  positively homogeneous  of degree $1$, we get that $\mathcal L(x,\cdot)$ is a proper map from $T_xM\setminus\{0\}$ to $T^*_xM\setminus\{0\}$. Being $n\geq 3$,  $T^*_xM\setminus\{0\}$ is simply connected and by \cite[Th. 1.8]{AmbPro93}, $\mathcal L(x,\cdot)$ is a bijection and therefore $\mathcal L$ is a diffeomorphism  from $TM\setminus 0$ onto $T^*M\setminus 0$.
\end{proof}
\begin{proof}[Proof of Theorem~\ref{geosforsplitting}]
We observe first that 
\beq\label{Lfurba}L(\tau, v)=-\Lambda\left(\tau-\frac{ B(v)}{\Lambda}\right)^2+\frac{B^2(v)}{\Lambda}+F^2(v).
\eeq 
	
(1) 
By considering variational vector fields $Z$ which are of the type $(Y,0)$   we deduce,  by a standard argument,  that \eqref{Killingconst} is  satisfied on $[a,b]$ for some constant $c_\gamma$.

(2) Let $I$ be an interval where $\sigma$ is smooth and $Z$ a variational vector field along $\gamma$ of the type $(0,W)$, with $W$ having compact support in $I$. Then, in local natural coordinates $(x^1,\ldots, x^n, y^1, \ldots, y^n)$ of $T M$, $\sigma$  satisfies, in such interval,  \eqref{onM}.
At the instants $s_j$, $j\in \{0,\ldots,k\}$, where $\dot\sigma$ has a break,  by taking any vector $w_j \in T_{\sigma(s_j)}M$, $j\in\{1,\ldots,k-1\}$, and a  variational vector field $(0,W_j)$, such that $W_j(s_j)=w_j$ and $W\equiv 0$ outside a small neighbourhood of $s_j$, we get, using  that \eqref{onM} is satisfied both in $[s_{j-1}, s_j]$ and $[s_{j}, s_{j+1}]$, 
\[\frac{\partial H}{\partial y^i}(\dot\sigma^-(s_j))w_j^i=\frac{\partial H}{\partial y^i}(\dot\sigma^+(s_j))w_j^i,\]
hence $(\partial_y H)_{\dot\sigma^-(s_j)}=(\partial_y H)_{\dot\sigma^+(s_j)}$. From Lemma~\ref{legendre}, $\dot\sigma^-(s_j)=\dot\sigma^+(s_j)$. Thus, $\sigma$ is a $C^1$ curve on $[a,b]$ and, from \eqref{Killingconst}, also $\theta$ is a $C^1$ function. By putting \eqref{onM} in normal form  on each interval $I$ where $\sigma$ is smooth (recall again Lemma~\ref{legendre}) we  deduce that also the second derivatives of $\sigma$ must agree at the instants $s_j$ and then both $\sigma$ and $\theta$ are smooth curves.  
\end{proof}
If $B$ reduces to a one-form on $M$ then we can avoid to consider only regular curves on $M$ and  we can drop the assumption about the dimension of $M$. Indeed,  let us now define, for  $x_0, x_1\in  M$,  $\Omega_{x_0x_1}(M)$ as the set of the continuous,  piecewise smooth,  curves  $\sigma$ on $M$, parametrized on a given  interval $[a,b]\subset \R$ and connecting $x_0$ to $x_1$. The analogous space of paths between two points $p_0=(t_0,x_0)$, $p_1=(t_1,x_1)$ in $\tilde M$  is given by
\[\Omega_{p_0p_1}(\tilde M)=\Omega_{t_0t_1}(\R)\times\Omega_{x_0x_1}(M).\]
Taking into account that in this case $B$ is smooth as a map defined on $TM$ and $(\partial_yB)_v$ is equal to $B$ and then it is independent from $v$, we obtain, arguing  as in the proof of \cite[Theorem 2.13]{CapSta16} and using that the Legendre map associated to $F^2$, $v\in TM\mapsto (\partial_yF^2)_v\in T^*M$,  is  a homeomorphism diffeomorphism (whatever the dimension of $M$ is) we get the following:
\begin{prop}\label{Boneform}
Let $(\tilde M,L)$ be a stationary splitting Finsler spacetime where  $B$ is a one form on $M$. 	
A curve $\gamma\in \Omega_{p_0p_1}(\tilde M)$ is a critical point of $E$ in $\Omega_{p_0p_1}(\tilde M)$ if and only if Eqs. \eqref{Killingconst} and \eqref{onM} are satisfied. 
Any critical point $\gamma=(\theta,\sigma)$ of $E$ in $\Omega_{p_0p_1}(\tilde M)$ is 
then at least   $C^1$ on $[a,b]$ and   there exists a constant $C\in \R$ such that $L(\dot\gamma(s))=C$, for all $s\in [a,b]$. 

Moreover,   if $\gamma$ is   non-constant  then: $(a)$ if it is spacelike or lightlike (i.e.  $C\geq 0$) then $\dot\sigma$ never vanishes and $\gamma$ is smooth; $(b)$ if $\sigma$ is  constant equal to $x_0\in M$ on the whole interval $[a,b]$ then  $C<0$, $d\Lambda(x_0)=0$ and $\dot \theta$ is constant too; vice versa, if  $d\Lambda(x_0)=0$ then,  for each $\theta_0\in \R$ and $m\neq 0$, the curve $s\in[a,b]\mapsto (\theta_0+m(s-a), x_0)\in \tilde M$ is a  timelike geodesic.
\end{prop}
\bere
As in the previous  case about $B$, any geodesic of $(\tilde M,L)$ connecting $p_0$ to $p_1$, according to Definition~\ref{geodef}, is a   critical point of $E$ on $\Omega_{p_0p_1}(\tilde M)$. On the other hand, Proposition~\ref{Boneform} allows us to define geodesics also when $\dot\gamma\not\subset A$: for example, we can say that constant curves are geodesics as they satisfy Eqs. \eqref{Killingconst} and \eqref{onM}; analogously,  we can say when a flow line of $\partial_t$ is a geodesic and when a curve whose velocity vector is collinear to $\partial_t$ at some instants is a geodesic. Technically,    these cases are not covered by Definition~\ref{geodef} or by introducing geodesics as auto parallel curves with respect to a Finslerian connection because, whatever this connection is, it cannot be computed at vector $v\in T\tilde M$ where the fiberwise Hessian of $L$ is not defined.
\ere
\section{Fermat's principle}\label{fermatprinc}
Fermat's principle  characterizes  light rays as the critical points of the travel  time. Thanks to an ingenious  adaptation  of the notion of arrival time,  it remains valid in General Relativity   \cite{Kovner90, Perlic90}. Its extension to  Finsler spacetimes has been proved in \cite{perlick06} (see also \cite{torrome} for timelike geodesics).
In this appendix we prove that, under the assumption of Theorem~\ref{fermatmetrics} lightlike geodesics of a stationary splitting Finsler spacetime $(\R\times M,L)$ project on $M$ as pregeodesics of the corresponding Fermat metric. By pregeodesic, it is meant a curve that is an arbitrary reparametrization of a geodesic.
We point out that this result could be also somehow deduced from the Fermat's principle for Finsler spacetime in \cite{perlick06} 
since the    arrival time coincides, in our setting, with the value of the coordinate $t$ on $\R\times M$ at the final point of a lightlike curve and therefore, up to a constant,  with the length of its component w.r.t. one of the Fermat metrics. Anyway, we give here a simple proof which is tailored on the  
splitting structure of $\tilde M$ and gives also a precise information about the parametrization of the projection of the lightlike geodesic.

\begin{prop}
	Under the assumptions of Theorem~\ref{fermatmetrics}, let $(\tilde M, L, T^+\tilde M\setminus \mathcal T)$ (resp. $(\tilde M, L, T^-\tilde M\setminus \mathcal T)$) be a stationary splitting Finsler spacetime. A curve $\gamma:[a,b]\rightarrow\tilde M$, $\gamma=(\theta,\sigma)$, is a lightlike geodesic  if and only if $\sigma$ is a (non-constant) pregeodesic of the Fermat metric $F_B$ (resp. $F^-_B$) parametrized with $G(\dot{\sigma})=-c_\gamma$ (resp. $G(\dot{\sigma})=c_\gamma$), where $G$ is defined in \eqref{Gfinsler}, and $\theta(s)=\theta(a)+\int_{a}^{s}F_{B}(\dot\sigma)d\tau$ (resp. $\theta(s)=\theta(a)-\int_{a}^{s}F^-_{B}(\dot\sigma)d\tau$).
\end{prop}
\begin{proof}
From \eqref{Lfurba}, we get that for a lightlike curve
$\big(\dot{\theta}\Lambda-B(\dot{\sigma})\big)^2=G^2(\dot{\sigma})$. Moreover, as  $\gamma$ is lightlike, $\dot{\sigma}(s)\neq 0$ for all $s\in [a,b]$. Therefore, from Remark~\ref{causalgeo} and observing that any curve $(\theta,\sigma)$ such that $\theta(s)=\theta(a)+\int_{a}^{s}F_{B}(\dot\sigma)d\tau$ (resp. $\theta(s)=\theta(a)-\int_{a}^{s}F^-_{B}(\dot\sigma)d\tau$) is lightlike (recall Proposition~\ref{samecausality}) we have that the statement is equivalent to prove that  \eqref{onM} is the equation of geodesics of $(M, F_B)$ (resp. $(M,F_B^-)$) parametrized with $G(\dot{\sigma})=\mathrm{const.}$ (such a constant is equal to $-c_\gamma$ for $(\tilde M, L, T^+\tilde M\setminus \mathcal T)$ and  to $c_\gamma$ for $(\tilde M, L, T^-\tilde M\setminus \mathcal T)$, $c_\gamma$ being negative in the first case and positive in the second one). 
Notice that \eqref{onM} is equivalent to
\bmln
-\frac{c_\gamma^2}{2\Lambda^2(\sigma)}\frac{\partial\Lambda}{\partial x^i}(\sigma)-c_\gamma\frac{\partial}{\partial x^i}\left(\frac{B}{\Lambda}\right)(\dot\sigma)+c_\gamma\frac{d}{ds}\left(\frac{\partial}{\partial y^i}\left(\frac{B}{\Lambda}\right)(\dot\sigma)\right)\\
+\frac 1 2\frac{\partial \tilde G^2}{\partial x^i}(\dot\sigma)-\frac 1 2\frac{d}{ds}\left(\frac{\partial \tilde G^2}{\partial y^i}(\dot\sigma)\right)=0,\quad i=1,\ldots,n
\emln
where $\tilde G^2=G^2/\Lambda$. Hence, if $G(\dot{\sigma})=-c_\gamma$ (the other case $G(\dot{\sigma})=c_\gamma$ is analogous) we get equivalently 
\bmln
-c_\gamma\left(\tilde G(\dot{\sigma})\frac{\partial}{\partial x^i}\left(\frac{1}{\sqrt{\Lambda}}\right)(\sigma)+\frac{1}{\sqrt{\Lambda(\sigma)}}\frac{\partial \tilde G}{\partial x^i}(\dot\sigma)-\frac{d}{ds}\left(\frac{1}{\sqrt{\Lambda(\sigma)}}\frac{\partial \tilde G}{\partial y^i}(\dot\sigma)\right)\right.\\
\quad\quad\left.+\frac{\partial}{\partial x^i}\left(\frac{B}{\Lambda}\right)(\dot\sigma)-\frac{d}{ds}\left(\frac{\partial}{\partial y^i}\left(\frac{B}{\Lambda}\right)(\dot\sigma)\right)\right)=0,\quad i=1,\ldots,n
\emln
which is, up to the constant factor $c_\gamma$, the equation of pregeodesics of
$F_B$, in local natural coordinates on $TM$.   
\end{proof}
\bere
In particular when $B$ is a one-form on $M$, lightlike future-pointing geodesics project to pregeodesics of $F_B$ and past-pointing ones to  pregeodesics of $F_B^-$.   
\ere


\begin{thebibliography}{10}



\bibitem{amir}
A.B. Aazami,  M.A., Javaloyes, 
{\em Penrose's singularity theorem in a Finsler spacetime}, Classical  Quantum Gravity {\bf 33}, 025003 (2016)

\bibitem{amir2}
A.B. Aazami, C.R. Keeton, A.O. Petters,
{\em Lensing by Kerr black holes. II: Analytical study of quasi-equatorial lensing observables},  J. Math. Phys. \textbf{52}, 102501 (2011)


\bibitem{AbCaLa88}
M. Abramowicz, B. Carter, J. Lasota, 
\emph{ Optical reference geometry
for stationary and static dynamics}, General Relativity and Gravitation  \textbf{20}, 1173--1183 (1988)


\bibitem{AmbPro93}
A. Ambrosetti and G. Prodi, {\em A Primer of Nonlinear Analysis}, Cambridge Studies in Advanced Mathematics, Cambridge University Press, Cambridge, 1993

\bibitem{AnInMa93}
P.L. Antonelli, R.S. Ingarden, M. Matsumoto, \emph{The Theory of Sprays and {F}insler Spaces with Applications in {P}hysics and {B}iology}, Fundamental Theories of Physics, Kluwer Academic Publishers Group, Dordrecht, 1993


\bibitem{Chern} 
D. Bao, S.-S.  Chern,  Z. Shen, 
\emph{An Introduction to Riemann-Finsler Geometry}, Graduate Text in Mathematics \textbf{200}, Springer-Verlag, New York, 2000

\bibitem{Bar12}
E. Barletta, S. Dragomir, 
{\em Gravity as a Finslerian metric phenomenon},
Found. Physics 42, 436--453 (2012)


\bibitem{BilJav08}
L. Biliotti and M.A. Javaloyes, {\em $t$-periodic light rays in conformally stationary spacetimes via {F}insler geometry}, Houston J. Math. {\bf 37}, 127--146 (2011)
\bibitem{Bogosl77a}
G.Y. Bogoslovsky, {\em A special-relativistic theory of the locally anisotropic space-time. I: The metric and group of motions of the anisotropie space of events.}, Il Nuovo Cimento B {\bf 40},
99--115 (1977)
\bibitem{Bogosl77b}
G.Y. Bogoslovsky, {\em A special-relativistic theory of the locally anisotropic space-time. II: mechanics and eleetrodynamics in the anisotropic space}, Il Nuovo Cimento B {\bf 40},
116--134 (1977)
\bibitem{Bogosl94}
G.Y. Bogoslovsky, {\em A viable model of locally anisotropic space-time and the Finslerian generalization of the relativity theory},  Fortschritte der Physik {\bf 42}, 143--193 (1994)


\bibitem{Brandt1999}
H.E. Brandt, {\em Finslerian fields in the spacetime tangent bundle}, 
Chaos Solitons and Fractals {\bf 10}, 267--282 (1999)

\bibitem{Cap10}
E. Caponio,	{\em The index of a geodesic in a {R}anders space and some remarks
	about the lack of regularity of the energy functional of a
	{F}insler metric}, Acta Math. Acad. Paedagogicae Nyh\'azi. (N.S.) {\bf 26},
265--274 (2010)


\bibitem{CapGerSan12}
E.~Caponio, A.~Germinario and M.~S{\'a}nchez, {\em Convex regions of
  stationary spacetimes and Randers spaces. Applications to lensing and
  asymptotic flatness}, J. Geom. Anal. {\bf 26}, 791--836 (2016)  

\bibitem{CaJaMa11}
E. Caponio, M.A. Javaloyes, A. Masiello, 
{\em On the energy functional on {F}insler manifolds and applications to stationary spacetimes}, Math. Ann. \textbf{351}, 365--392 (2011)


\bibitem{CaJaMa10}
E. Caponio, M.A. Javaloyes, A. Masiello, 
{\em Finsler geodesics in the presence of a convex function and their applications}, J. Phys. A {\bf 43}, 135207, 15 (2010)

  
\bibitem{CaJaMa10a} 
E. Caponio, M.A. Javaloyes, A. Masiello, 
{\em Morse theory of causal geodesics in a stationary spacetime via {M}orse theory of geodesics of a {F}insler metric}, Ann. Inst. H. Poincar{\'e} Anal. Non Lin{\'e}aire {\bf 27}, 857--876 (2010)

\bibitem{CaJaMa13}
E. Caponio, M.A. Javaloyes, A. Masiello, 
{\em Addendum to ``{M}orse theory of causal geodesics in a stationary spacetime via {M}orse theory of geodesics of a {F}insler metric'' [{A}nn. {I}. {H}. {P}oincar{\'e}---{AN} 27 (3) (2010) 857--876]}, Ann. Inst. H. Poincar{\'e} Anal. Non Lin{\'e}aire {\bf 30}, 961--968 (2013)


\bibitem{CJS} 
E. Caponio, M.A. Javaloyes, M. S\'anchez, 
\emph{On the interplay between Lorentzian Causality and Finsler metrics of Randers type}, Rev. Mat. Iberoamericana \textbf{27}, 919--952 (2011)

\bibitem{CJS2} 
E. Caponio, M.A. Javaloyes, M. S\'anchez, 
{\em Wind Finslerian structures: from Zermelo's navigation to the causality of spacetimes}, \href{http://arxiv.org/abs/1407.5494}{arXiv:1407.5494 [math.DG]} (2015)

\bibitem{CaJaPi10}
E. Caponio, M.A. Javaloyes, P. Piccione
{\em Maslov index in semi-{R}iemannian submersions}, Ann. Global Anal. Geom., \textbf{38}, 57--75 (2010)
  


\bibitem{CapSta16}
E. Caponio, G. Stancarone,
{\em Standard static {F}insler spacetimes},
Int. J. Geom. Methods Mod. Phys., {\bf 13}, 1650040, 25 (2016)
  


\bibitem{Colladay}
D. Colladay, P.  McDonald, 
{\em Classical Lagrangians for momentum dependent Lorentz violation},
Phys. Rev. D {\bf 85}, 044042 (2012)

\bibitem{FoGiMa95}
D. Fortunato, F. Giannoni, A.  Masiello, 
{\em A {F}ermat principle for stationary space-times and applications to light rays}, J. Geom. Phys. {\bf 15}, 159--188 (1996)

\bibitem{FlHeSa13}
J.L. Flores, J. Herrera and M. S{\'a}nchez, 
{\em Gromov, {C}auchy and causal boundaries for {R}iemannian, {F}inslerian and {L}orentzian manifolds},
Mem. Amer. Math. Soc. {\bf 1064}, vi+76 (2013)


\bibitem{FusPab16}
A. Fuster, C.U. Pabst,
{\em Finsler pp-waves}, Phys.Rev. D {\bf 94}, 104072 (2016)
 
\bibitem{torrome}
R. Gallego Torrom\'{e}, P. Piccione, H. Vit\'{o}rio,
{\em On Fermat's principle for causal curves in time oriented Finsler spacetimes},
J. Math. Phys. \textbf{53},  123511 (2012)

\bibitem{GarSen03}
A. Garc\'\i a-Parrado, J.M.M. Senovilla,
{\em Causal relationship: a new tool for the causal characterization of {L}orentzian manifolds},
Classical Quantum Gravity {\bf 20}, 625--664 (2003)

\bibitem{GiaPic99}
F. Giannoni, P. Piccione, {\em An intrinsic approach to the geodesical connectedness of stationary {L}orentzian manifolds}, Commun. Anal. Geom.  {\bf 7},  157--197 (1999)
	

\bibitem{GiGoPo07}
{G.W. Gibbons, J. Gomis, C.N. Pope},
{\em General very special relativity is Finsler geometry},
Phys. Rev. D {\bf 76}, 081701 (2007)


\bibitem{Gib09}
G.W. Gibbons, C.A.R. Herdeiro, C.M. Warnick and M.C. Werner,
{\em Stationary metrics and optical Zermelo-Randers-Finsler geometry}, Phys. Rev. D {\bf 79}, 044022 (2009) 


\bibitem{GiLiSi2007}
F. Girelli, S. Liberati, L. Sindoni, {\em Planck-scale modified dispersion relations and Finsler geometry}, Phys. Rev. D {\bf 75},  064015 (2007),

\bibitem{HehObu03}
F.W. Hehl, Y.N. Obukhov,  {\em Foundations of Classical Electrodynamics. Charge, Flux, and Metric}, Birkh\"auser Boston, Inc., Boston, MA, 2003

	


\bibitem{HohPfe17}
M. Hohmann, C. Pfeifer,
{\em Geodesics and the magnitude-redshift relation on cosmologically symmetric Finsler spacetimes}, Phys. Rev. D {\bf 95}, 104021 (2017)


\bibitem{IyeHan09}
S.V.Iyer, E.C. Hansen, 
{\em Light's bending angle in the equatorial plane of a Kerr black hole},
Phys. Rev. D {\bf 80}, 124023 (2009)
\bibitem{IyeHan09b}
S.V.Iyer, E.C. Hansen,
{\em Strong and weak deflection of light in the equatorial plane of a Kerr black hole}, arXiv:0908.0085 [gr-qc] (2009)

\bibitem{java}
M.A. Javaloyes, {\em Anisotropic tensor calculus},
arXiv:1602.05492v1 [math.DG] (2016)



\bibitem{JavS08}
M.A. Javaloyes, M.  S{\'a}nchez, \emph{A note on the existence of standard splittings for conformally stationary spacetimes}, Classical  Quantum Gravity, \textbf{25}, 168001, 7, (2008).


\bibitem{JavSan14}
M.A. Javaloyes, M.  S{\'a}nchez, \emph{On the definition and examples of Finsler metrics},
Ann. Scuola Norm. Sup. Pisa Cl. Sci. (4) {\bf XIII},  813--858 (2014)
  

\bibitem{Koste}
V.A. Kosteleck\'{y}, 
{\em Riemann-Finsler geometry and Lorentz-violating kinematics},
Phys. Lett. B  {\bf 701}, 137 (2011)

\bibitem{Russell}
V.A. Kosteleck\'{y}, N. Russell, R. Tso,
{\em Bipartite Riemann-Finsler geometry and Lorentz violation},
Phys. Lett. B {\bf 716}, 470 (2012)


\bibitem{KoStSt12}
A.P. Kouretsis, M. Stathakopoulos, P.C. Stavrinos, 
{\em Covariant kinematics and gravitational bounce in Finsler space-times}, Phys. Rev. D {\bf 86}, 124025 (2012)

\bibitem{Kovner90}
I. Kovner, 
{\em Fermat principles for arbitrary space-times}, Astrophys. J. {\bf 351}, 114--120 (1990)
	

\bibitem{LPH}
C. L\"ammerzahl, V. Perlick, W. Hasse,
{\em Observable effects in a class of spherically symmetric static Finsler spacetimes},
Phys. Rev. D {\bf 86}, 104042 (2012)

\bibitem{LiChang} X. Li, Z. Chang, {\em Exact solution of vacuum field equation
in Finsler spacetime}, Phys. Rev. D {\bf 90},  064049 (2014)

\bibitem{Lovas04}
R.L. Lovas, \emph{On the {K}illing vector fields of generalized metrics},
SUT J. Math. {\bf 40}, 133--156 (2004)

\bibitem{mas}
A. Masiello,
{\em Variational Methods in Lorentzian Geometry}.
Pitman Research Notes in Mathematics Series \textbf{309},
Longman, London, 1994



\bibitem{MeSzTo03}
T. Mestdag, J.Szilasi, V. T{\'o}th,
{\em On the geometry of generalized metrics}, Publ. Math. Debrecen {\bf 62}, 511--545 (2003)


\bibitem{Minguzzi} 
E. Minguzzi,  
\emph{Light cones in Finsler spacetime}, Comm. Math. Phys. \textbf{334}, 1529--1551 (2015)

\bibitem{Minguz15b}
E. Minguzzi, 
{\em Raychaudhuri equation and singularity theorems in {F}insler spacetimes},
Classical Quantum Gravity {\bf 32}, 185008, 26 (2015)


\bibitem{Minguzzi Sanchez}
E. Minguzzi, M. S\'anchez, \emph{The Causal Hierarchy of Spacetimes. In Recent developments in pseudo-{R}iemannian geometry.}
Eur. Math. Soc., Z{\"u}rich, 299--358 (2008)

\bibitem{O'Neil}
B. O'Neill, 
\emph{Semi-Riemannian Geometry}, Pure and Applied Mathematics \textbf{103}, Academic Press, Inc., New York, 1983

\bibitem{OoYaIs17}
T. Ootsuka, R.  Yahagi, M. Ishida,
{\em Killing symmetry on the {F}insler manifold},
Classical Quantum Gravity {\bf 34}, 095002, 22 (2017)

\bibitem{papa}
G. Papagiannopoulos, S. Basilakos, A. Paliathanasis, S. Savvidou, P. Stavrinos,
{\em Finsler-Randers Cosmology: dynamical analysis and growth of matter perturbations},
Classical and Quantum Gravity {\bf 34}, 225008 (2017)

\bibitem{Perlic90}
V. Perlick, 
{\em On {F}ermat's principle in {G}eneral {R}elativity. {I}.\ {T}he general case}, Classical Quantum Gravity \textbf{7}, 1319--1331 (1990) 
	



\bibitem{perlick06}
V. Perlick, 
{\em Fermat principle in {F}insler spacetimes}, Gen. Relativity Gravitation {\bf 38}, 365--380 (2006)

\bibitem{PfeWol11}
C. Pfeifer, M.N.R. Wohlfarth,  
{\em Causal structure and electrodynamics on Finsler spacetimes},
Phys. Rev. D  {\bf 84}, 044039 (2011)


\bibitem{Rander41}
G. Randers,  {\em On an asymmetrical metric in the fourspace of {G}eneral {R}elativity}, Phys. Rev. (2) {\bf 59}, 195--199 (1941)



\bibitem{Russel15}
N. Russell,
{\em Finsler-like structures from Lorentz-breaking classical particles},
Phys. Rev. D {\bf 91}, 045008 (2015)

\bibitem{Rutz}
S.F. Rutz, {\em A Finsler generalisation of Einstein's vacuum field equations}, 
General Relativity and Gravitation {\bf 25},  1139-1158 (1993)

\bibitem{shreck}
M. Schreck, 
{\em Classical Lagrangians and Finsler structures for the nonminimal fermion sector of the standard model extension},
Phys. Rev. D {\bf 93}, 105017 (2016)

 

\bibitem{stavac} P. Stavrinos, S. Vacaru, {\em Cyclic and ekpyrotic
universes in modified Finsler osculating gravity on tangent Lorentz bundles},
Classical Quantum Gravity {\bf 30}, 055012 (2013)

\bibitem{Szilas03}
J. Szilasi, \emph{A Setting for Spray and {F}insler Geometry}, Handbook of {F}insler geometry. Kluwer Acad. Publ., Dordrecht, 2003


\bibitem{vacaru12} S. Vacaru, {\em Principles of Einstein-Finsler gravity and
perspectives in modern cosmology}, Int. J. Mod. Phys. D {\bf 21}, 1250072 (2012)


\bibitem{Werner12}
M.C. Werner,
{\em Gravitational lensing in the Kerr-Randers optical geometry},
General Relativity and Gravitation {\bf 44}, 3047--3057 (2012)


\end{thebibliography}
\end{document}